\documentclass[final]{amsart}

\usepackage{geometry} 
\usepackage{graphicx,latexsym, color}
\usepackage{amssymb}
\usepackage{amsrefs}

\newtheorem{theorem}{Theorem}[section]
\newtheorem{proposition}{Proposition}[section]
\newtheorem{corollary}{Corollary}[section]
\newtheorem{lemma}{Lemma}[section]
\newtheorem{remark}{Remark}[section]
\newtheorem{definition}{Definition}[section]

\numberwithin{equation}{section}
\newcommand{\Frac}[2]{\frac{\textstyle #1}
                           {\textstyle #2}}

\newcommand{\dist} {{\rm dist}}
\newcommand{\diag} {{\rm diag}}
\newcommand{\be}{\begin{equation}}
\newcommand{\ee}{\end{equation}}
\newcommand{\beq}{\begin{equation}}
\newcommand{\eeq}{\end{equation}}
\newcommand{\ba}[1] {\begin{array}{ #1 }}
\newcommand{\ea}{\end{array}}
\newcommand{\bc}{\begin{center}}
\newcommand{\ec}{\end{center}}
\newcommand{\bt}[1] {\begin{tabular} {#1} }
\newcommand{\et}{\end{tabular}}
\newcommand{\tpi} { \tilde{\pi} }

\def\vU {\vec{U}}

\newcommand{\cA} {{\mathcal A}}
\newcommand{\cC} {{\mathcal C}}

\newcommand{\cF} {{\mathcal F}}
\newcommand{\oG} {{\mathcal G}}
\def\mcG{{\bf R}}

\newcommand{\cJ} {{\mathcal J}}
\newcommand{\cK} {{\mathcal K}}
\newcommand{\cL} {{\mathcal L}}
\newcommand{\cM} {{\mathcal M}}
\newcommand{\cN} {{\mathcal N}}

\newcommand{\cR} {{\mathcal R}}

\newcommand{\cV} {{\mathcal V}}

\newcommand{\cNon} {{\mathcal X}}

\newcommand{\eps} {{\epsilon}}
\newcommand{\vp} {{\vec{p}}}
\newcommand{\vq} {{\vec{q}}}
\def\tp{{\tilde{p}}}
\def\tq{{\vec{q}}}
\def\vchi{{\vec{\chi}}}
\def\vxi{{\vec{\xi}}}
\def\bR{{\mathbb R}}

\def\bC{\mathbb C}

\def\tL{\tilde{L}}

\def\oalpha{{\gamma}}

\newcommand{\leavethisout}[1] {}

\begin{document}

\title{Adiabatic stability under semi-strong interactions:\\ The weakly damped regime}
\author{Thomas Bellsky}
\address{School of Mathematics and Statistics, Arizona State University, Tempe, AZ 85287}
\email{bellskyt@asu.edu}

\author{Arjen Doelman}
\address{Mathematisch Instituut, Universiteit Leiden,
P.O. Box 9512, 2300 RA Leiden, Netherlands}
\email{doelman@math.leidenuniv.nl}

\author{Tasso J. Kaper}
\address{Department of Mathematics \& Center for BioDynamics,
Boston University, 111 Cummington Street,
Boston, MA 02215}
\email{tasso@math.bu.edu}

\author{Keith Promislow}
\address{Department of Mathematics, Michigan State University, East Lansing, MI 48824}
\email{kpromisl@math.msu.edu}




\begin{abstract}
We rigorously derive  multi-pulse interaction laws for the semi-strong interactions in a family of singularly-perturbed
and weakly-damped reaction-diffusion systems in one space dimension. Most significantly, we show the existence of a manifold of 
quasi-steady $N$-pulse solutions and identify a ``normal-hyperbolicity'' condition which balances the 
asymptotic weakness of the linear damping against the algebraic evolution rate of the multi-pulses. Our main result is the
adiabatic stability of the manifolds subject to this normal hyperbolicity condition. More specifically,
the spectrum of the linearization about a fixed $N$-pulse configuration contains essential spectrum that is 
asymptotically close to the origin as well as {\em semi-strong} eigenvalues which move at leading order as the pulse positions
evolve. We characterize the semi-strong eigenvalues in terms of the spectrum of an explicit $N\times N$ matrix, and rigorously 
bound the error between the $N$-pulse manifold and the evolution of the full system, in a polynomially weighted space, so long as the semi-strong spectrum 
remains strictly in the left-half complex plane, and the essential spectrum is not too close to the origin.
\end{abstract}

\maketitle

\pagestyle{myheadings}
\thispagestyle{plain}
\markboth{T. Bellsky, A. Doelman, T. Kaper, and K. Promislow}{Adiabatic Stability under Semi-Strong Interactions}

\section{Introduction} 
It is not uncommon that evolutionary partial differential equations have finite dimensional sub-manifolds which are approximately invariant, and
robustly stable, in the sense that initial data which starts close to the sub-manifold remains close. For these proximal orbits, the full dynamics typically can 
be reduced to a finite dimensional system, often with nontrivial dynamics of its own.  Asymptotic analysis is adept at constructing such manifolds;
determining the stability of the sub-manifold is more problematic, as the natural object which arises in the linear stability theory may have non-trivial time-dependence. The road from a time-dependent linear operator to the properties of its semi-group is long and hard.  A tractable sub-case arises when the flow on the sub-manifold is slow in comparison to the exponential rates which characterize the decay of proximal orbits towards the sub-manifold; in this case we say the sub-manifold is normally hyperbolic. It is natural to investigate the case in which the normal hyperbolicity is lost due to a decrease in the exponential decay rates of the sub-manifold. We address this question within the context of semi-strong multi-pulse interactions in a weakly-damped reaction diffusion system.
  
The study of pulse interaction in diffusive systems has a long history, in particular the activator-inhibitor systems modeled by the 
Gierer-Meinhardt \cite{GM} and the Gray-Scott \cite{GS} equations have spawned a substantial literature.  
These reaction-diffusion systems are comprised of two chemicals which feed an autocatalytic reaction that drives pattern formation.
 Our analysis is particularly motivated by the Gray-Scott Model
\beq
\label{eqGS}
\begin{aligned}
 U_t =& U_{xx} + A(1 - U) - UV^2, \\
V_t =& DV_{xx} - BV + UV^2,
\end{aligned}
\eeq
where $U$ and $V$ denote the concentrations of the chemical species. The semi-strong regime of the Gray-Scott equations, presented in \cites{DGK-98, GS1, DGK-02}, occurs when the $U$ component experiences strong diffusion, but is weakly damped and strongly forced. In particular, there is a (nontrivial) balance which maintains the  two components at an $O(1)$ level within the spatialdomain where the slowly diffusing component is present. This leads to a rich stability structure in which the localized species manifests a long-range interaction through the delocalized species. This simple system affords an ideal arena
for the study of reductions of infinite dimensional dynamical systems to finite dimensional sub-systems.  This paper studies a generalization of the
rescaled system. This generalization maintains the aforementioned balance while enjoying a variable rate of linear damping and a homogeneous polynomial nonlinearity in $\vU=(U,V)^T$, 
\beq
\label{UV}
\left.
\begin{aligned}
 U_t &=& \epsilon^{-2}U_{xx}- \epsilon^{\alpha}\mu U-\epsilon ^{-1}U^{\alpha_{11}}V^{\alpha_{12}}+\eps^{\alpha/2}\rho, \\
 V_t &=& V_{xx}-V+U^{\alpha_{21}}V^{\alpha_{22}},\hspace{1.15in}
\end{aligned} \right\} =: \cF(\vU),
\eeq
where $ \alpha_{11},$ $\alpha_{21} \geq 0$ and $\mu >0$, $\alpha > 0$, $\rho\geq 0$, $\alpha_{12} >1$ and $\alpha_{22}> 1$. 
For  $0<\epsilon \ll 1$ the system affords a natural competition between  long-range, $U$, and short-range, $V$, interactions. 
Such competitions are wide-spread in physical settings, arising for example from the  balance between entropic and electrostatic interactions in ionic solutions which 
drive morphology generation in solvated charged polymers, \cites{GP-11, Polymer-12, DaiP-12, DP-12}. 

We address the existence, dynamics, and particularly the adiabatic stability of $N$-pulses in the semi-strong interaction regime. 
This regime features pulse-like structures in the localized species whose positions, amplitudes, and stability all evolve at leading order subject to an 
effective mean-field generated by the pulses coupling to the delocalized,  rapidly-diffusing species. The novelty of our study lies in the asymptotic approach 
of the essential spectrum to the origin,  characterized by the weak linear-damping rate, $\eps^\alpha\mu$, which may lead to a loss of normal hyperbolicity (adiabaticity) 
as the slow decay associated to the essential spectrum competes with the perturbations generated by the pulse evolution.

For fixed $N\in\mathbb{N}_+$, we rigorously derive the existence and adiabatic stability of semi-strong $N$-pulse configurations. 
This study requires a careful  analysis of the  linearization about the $N$-pulse configurations, including resolvent and semi-group estimates as well as a characterization of the point spectrum.  Generalizing prior results, \cites{GS1, IndRD, ironward},  we show 
that there is a set of point spectra,  the {\em semi-strong spectrum},  which evolve at leading order in conjunction with the localized pulse configuration,
and can be characterized in terms of the eigenvalues of an explicit $N\times N$ matrix.  Extending these results, 
we rigorously show that so long as the finite-rank spectrum remains uniformly within the left-half plane and the linear damping is stronger
than a critical value, then the $N$-pulses generate an adiabatically stable manifold which affords a leading order description of the 
pulse interactions. A key element of the analysis is the development of semi-group estimates, via a renormalization group (RG) approach, 
corresponding to the {\sl evolving} pulse configurations. This result requires a form of normal hyperbolicity for the manifold of semi-strong $N$-pulses, 
balancing the flow on the manifold against the weak linear decay, see Theorem\, \ref{thm:main} and following discussion.

There is a developed literature on the stability of viscous shocks and traveling waves, for which the essential spectrum touches 
the origin, see \cites{SS-99, SS-00, Zumbrun-02, GUS-04, Howard-07, Zumbrun-11} and references therein. 
We emphasize two important distinctions between our results and these two bodies of work. 
For the nonlinear conservation laws, \cites{Zumbrun-02, Howard-07, Zumbrun-11}, the stability estimates do not 
``close'' in the following sense: initial perturbations are considered in a space, $L^1(\bR)$ for example, which is not controlled by 
the estimates at later times -- the decay is in a norm which does not control the initial perturbation. As a consequence the process cannot be 
iterated; one cannot restart the perturbation analysis at a later time in the flow. For the traveling waves under essential bifurcation, \cites{SS-99, SS-00, GUS-04},
the solution oscillates temporally in a neighborhood of a fixed structure: the system is controlled by a single, temporally fixed linearized operator, thereby avoiding the issue of 
competition between decay rates and secular forcing arising from a time-dependent linearization. Neither family of results extends trivially to encompass the semi-strong interactions, in which the underlying structure, the $N$-pulse positions and amplitudes, evolve at leading order generating concomitant changes in the associated linearized operators. 
In the setting of \eqref{UV}, one is faced either with developing semi-group estimates on a time-dependent family of operators, as in \cite{ScheelWright}, or in the approach 
we follow here: taking the linearized operators to be piece-wise constant in time with a renormalization of the flow at each jump in the linearization.  
This later, iterative approach requires estimates on the decay in the same norm used to control the initial perturbations.

\subsection{Prior results on semi-strong interactions}

The semi-strong interaction regime is intermediate between the weak interaction regime and the strong interaction regime. 
In the weak regime, the pulses are localized in each component, which returns to its equilibrium value between adjacent pulses.
The mutual interaction of localized structures is exponentially weak in the pulse separation distance, and consequently there is no 
leading order influence of pulse location on the shape or the stability of the pulses.   The weak interaction regime has been well-studied in reaction-diffusion systems, see \cites{EiRD, EiPP, RGProm, Sandstede}.  
The strong interaction regime arises when the pulses are sufficiently proximal that the values of their localized components do not return to
equilibrium values between the pulses. This leads to self-replication, collision, annihilation, and other strongly non-adiabatic behaviors.  
There has been little theoretical investigation of the  strong interaction regime, which is typically investigated using numerical techniques.
In the semi-strong regime the pulses have both localized and delocalized components, with the delocalized components varying slowly over the support
of the localized components. This regime has been studied both formally \cites{GS1, S-SIRD, ironward} and rigorously \cites{RGG-M, RGPVH}. 

These previous works have largely focused on the Gierer-Meinhardt and the Gray-Scott models.  In \cite{ironward}, a formal study of the $N$-pulse semi-strong interaction regime for the generalized Gierer-Meinhardt model was presented. In particular, expressions for 
the semi-strong spectrum and the ordinary differential equations for the dynamics of $N$-pulse configurations were derived.  
In \cite{S-SIRD}, a general system that includes both the Gierer-Meinhardt and the Gray-Scott model was studied. The semi-strong two-pulse interaction was investigated, where formal results for the asymptotic stability were determined in particular regimes, along with ordinary differential equations governing the dynamics of pulse positions. The extension of these results to the $N$-pulse was also discussed. However, the derivation of conditions under which the $N$-pulse manifold is adiabatically attractive under the full flow of the PDE is outside the scope of the prior work.

In \cite{RGG-M}, the $2$-pulse semi-strong interaction regime was rigorously studied for the regularized Gierer-Meinhardt model,
which has a strong linear damping rate, extending the renormalization group approach developed in \cite{RGProm}  
to obtain appropriate semi-group estimates on families of weakly time-dependent linear operators. 
 In \cite{RGPVH}, the renormalization group approach was used to establish the adiabatic stability of $2$, $3$, and $4$-pulse configurations
within the semi-strong interaction regime for a three-component system with two inhibitor components and one strongly-linearly-damped activator component. 
The renormalization group techniques have also been used to study quasi-steady manifolds in noisy systems, \cite{Guha}, and coupled
dispersive-diffusive systems, \cite{MooreProm}.

The homogeneous nonlinearity considered in \eqref{UV} is characterized by the quantity
\be
\label{theta-def} \theta := \alpha_{11}-\Frac{\alpha_{12}\alpha_{21}}{\alpha_{22}-1}.
\ee
In particular, the admissibility conditions typically require $\theta\leq  0,$ 
see Remark \ref{R:theta}. The scaled Gray-Scott system, in \cites{DGK-98, GS1, DGK-02}, corresponds to $\theta=-1$, however its  linear damping rate, $\alpha=\frac32$, exceeds the critical value, $\alpha=\frac12$, permitted by Theorem\,\ref{thm:main}. This suggests that the evolution of its $N$-pulse solutions may not be describable by an $N$-dimensional system parameterized solely by the pulse locations. It may be necessary to include the impact of the essential spectrum -- which may manifest itself as a long, low shelf (asymptotically small in $L^\infty$ but large in $L^1$) which surrounds the pulse region. A detailed description of this supercritical regime is an intriguing open problem that is outside the scope of this work.

 \begin{figure}[ht]
\bc
\includegraphics[width=3.5in]{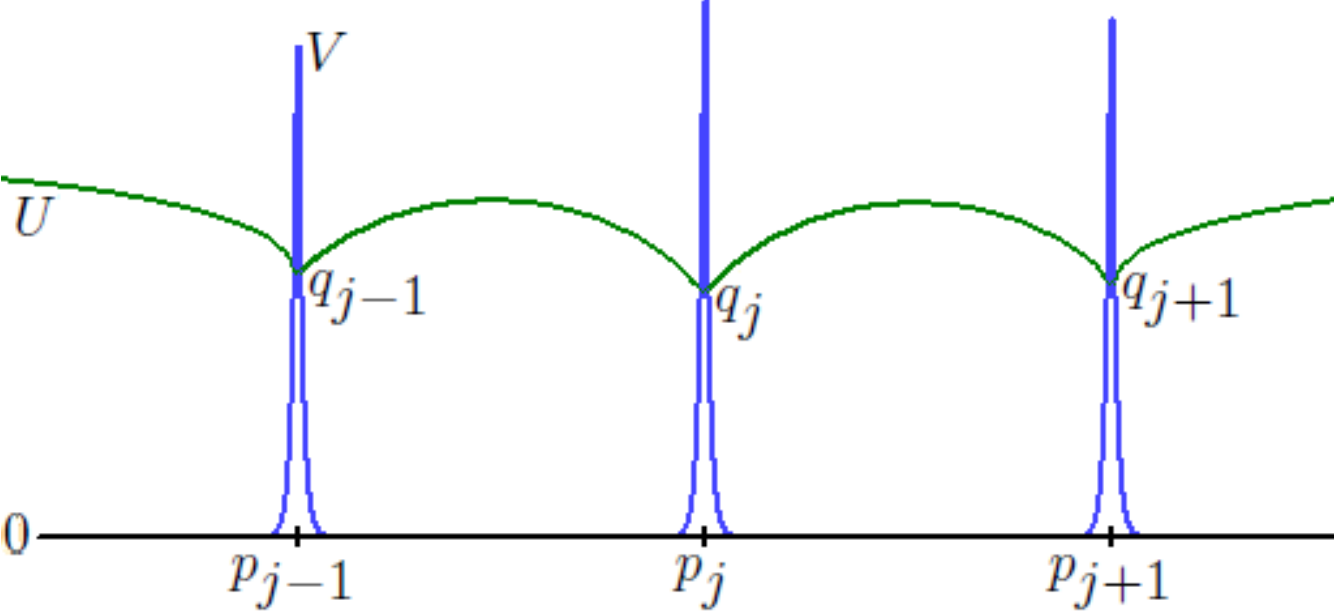}
\caption{A typical $N$-pulse configuration for the coupled system \eqref{UV}. The $V$ component is localized at the pulse positions $p_j$.  The $U$ component has an approximately constant value, $q_j$, on the narrow pulse intervals, and is slowly varying in between the pulses,  reaching its equilibrium value of $O(\eps^{-\alpha/2})$ as $x\to\pm\infty$.}
\label{f:N-pulse}
\ec
\end{figure}

\subsection{Presentation of the main results}
Our first result is the existence of the semi-strong $N$-pulse configurations parameterized by the localized pulse positions,
$\vp=(p_1, \cdots, p_N)^T\in \bR^N$. In general, a manifold formed by the graph of a function $\Phi=\Phi(\vp\,)$ is invariant under
a flow $U_t=\cF(U)$ on a Banach space $X$ if and only if
\beq\label{involution}  (I-\pi_\vp)\, \cF\!\left(\Phi(\vp\,)\right) =0, \eeq
where $\pi_\vp$ is the projection onto the tangent plane of the manifold at $\vp.$  
Indeed, the classic manifold formed by the translates of a traveling wave solution is invariant precisely because its wave-shape satisfies
the so-called traveling wave ODE,
$$c\Phi_x =\cF(\Phi),$$
for wave speed $c$. The manifold is parameterized by the translates $\cM:=\{\Phi(\cdot-p)\bigl| p\in\bR\}$ of $\Phi$, and 
the projection off of the tangent plane of $\cM$ has $\Phi_x$ in its kernel. In cases not arising from a natural symmetry of the flow, such as translation or rotation, construction of an invariant manifold is a non-trivial endeavor, see \cite{Bates2} for example. 

We do not perform this calculation, rather we construct a manifold with boundary, which is approximately invariant. 
More specifically, we construct a manifold for which the left-hand side of (\ref{involution}) is sufficiently small in an appropriate norm, 
and which has a thin, forward invariant neighborhood which attracts the flow from a thicker neighborhood -- at least up to times that the 
flow hits the boundary of the manifold. In Theorem\,\ref{thm:mean_field}, we reduce this construction to the solution of 
an $N\times N$ system of nonlinear equations  which connects the positions of the localized pulses, $\vp$, to the amplitudes 
of the delocalized component at the pulse position $q_i=U(p_i)$  for $i=1, \cdots, N$.  The family of solutions, $\vq=\vq\,(\vp\,)$ gives rise to 
the semi-strong $N$-pulse configuration
\beq
\label{Phiqs} \Phi(x;\vp\,) = \left(\ba c \Phi_1(x;\vp\,) \\ \Phi_2(x;\vp\,) \ea \right),
\eeq
defined in \eqref{Phi-def}.  In Proposition\,\ref{prop:4} we characterize the 
spectrum of the linearization $L_\vp$ of \eqref{UV} about $\Phi$, 
showing that the point spectrum consists of $N$ eigenvalues localized near the origin together with 
the {\em semi-strong} spectrum, denoted $\sigma_{\rm ss}(\vp\,)$, which
can move at leading order as the pulse positions $\vp$ evolve. Moreover, the locations of the
semi-strong spectrum are determined by the eigenvalues of an $N\times N$ matrix  $\cN_\lambda$ defined in (\ref{N}). 
This motivates the following definition.
\begin{figure}[ht]
\bc
\includegraphics[width=3.0in]{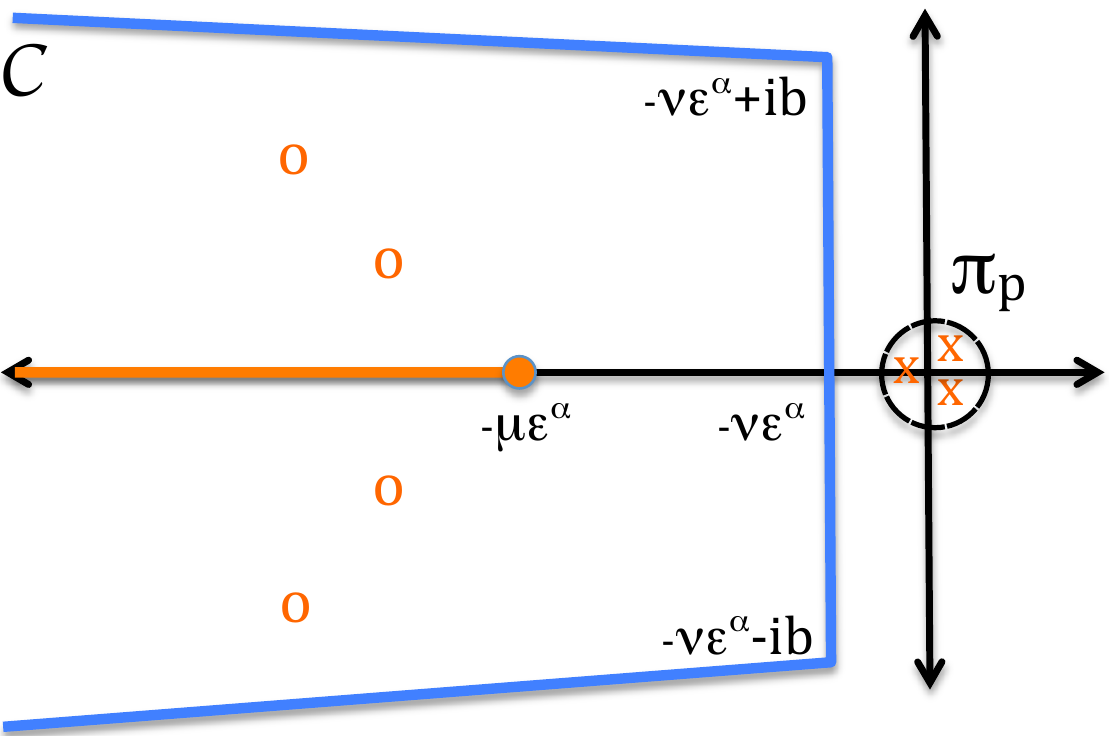}
\caption{\label{f:contour} The contour $\cC$ introduced in \eqref{contour}. The essential spectrum of the linearization $L_{\vec{p}}$ about
an $N$-pulse solution, see \eqref{lin-op}, extends from $(-\infty, -\mu\eps^\alpha]$ (orange line) while the point spectrum consists of 
$N$ broken translational eigenvalues near the origin (orange X), with associated spectral projection $\pi_\vp$, 
together with a collection of semi-strong spectrum (orange o) which move at leading order as the pulse configuration evolves.  
For an admissible pulse configuration, the semi-strong spectrum lies to the left of the contour $\cC$.}
\ec
\end{figure}

\begin{definition}\label{def-one}
Fix $\eps>0$.  An open, connected set $\cK\subset \bR^N$ is an admissible family of $N$-pulses if there exist $b, C, \ell>0$  and 
$\nu\in(0,\mu)$ such that
 the following hold.
 \begin{enumerate}
 \item  For all $\vp \in\cK$ the semi-strong spectrum $\sigma_{\rm ss}(\vp\,)$ lies to the
 left of the contour $\cC \subset \mathbb{C}$, of the form
\beq
\label{contour} \cC = \cC_\nu \cup \cC_{-}\cup \cC_{+},
\eeq
where $\cC_\nu=\left\{ -\eps^\alpha\nu+is \big{|}s \in [-b,b] \right\}$ and 
$\cC_{\pm}=\left\{-\eps^\alpha\nu \pm ib +se^{\pm \frac{i5\pi}{6}}\big{|}s \in [-\infty,0] \right\}$, see Figure\,\ref{f:contour}. 

\item
For all $\lambda\in\cC$ and $\vp\in\cK$,
\beq  \label{Nlam-est}
 \left| \left( I+\cN_\lambda\right)^{-1}\right| \leq C\left(1+\frac{\eps}{|k_\lambda|}\right)^{-1}.
 \eeq
 where $k_\lambda:=\eps\sqrt{\lambda+\eps^\alpha\mu}$ is a scaled distance of $\lambda$ to the branch point
 of the essential spectrum of $\cL_\vp.$ 
 \item
  $\cK\subset\cK_\ell$ where
 \beq \label{cKell-def}
 \cK_\ell :=\left\{\vp\in \bR^N \Bigl| \Delta p_i:= p_{i+1}-p_{i} \geq \ell |\ln\eps|\right\}, 
 \eeq
 for $i=1, \cdots, N-1.$
 \end{enumerate}
 \end{definition}

To each admissible family of pulse configurations $\cK\subset \mathbb{R}^N$, we associate the $N$-pulse manifold
\beq
\label{qsManifold} \cM:=\left\{\Phi \left(\cdot;\vp\,\right)\Bigl |\,\, \vp \in \cK \right\}.
\eeq
Our first result is that a non-trivial portion of the semi-strong $N$-pulse interaction regime is admissible if the 
corresponding single-pulse is linearly stable.

\begin{proposition} {\bf Admissibility.}
\label{p:admissible} Consider the system \eqref{UV}. Let the associated single-pulse solution defined by \eqref{Phi-def} be linearly stable, 
that is, if except for a simple translational eigenvalue at the origin, the point spectrum of the  linearization about a single-pulse lies 
to the left of a contour $\cC_{\nu,b}$ for some $\nu, b>0.$ Then for every $N\in \mathbb{N}_+$  there exists $\eps_0>0$ and
such that for all $\delta>0$ sufficiently large and for all $0<\eps<\eps_0$ the portion of the semi-strong domain
satisfying
\beq \label{cK-delta}
 \cK := \left\{ \vp\in \bR^N \Bigl| \, \Delta p_i \geq \delta \eps^{-(1+\alpha/2)}\,\,{\rm for}\,\,  i=1, \cdots, N-1 \right\},
 \eeq
 is admissible. 
\end{proposition}

\begin{remark} \label{R:theta}
The single-pulse can be linearly stable only if $\theta<0.$ For $\theta>0$ the linearization about the single pulse has a real eigenvalue
which lies between the origin and the ground-state eigenvalue $\lambda_0>0$ of the operator $L_0$ defined in \eqref{L0-def}.
However there are semi-strong two-pulse regimes which are admissible, even when the underlying one-pulse is not linearly stable,
see \cite{GS1}. The restriction on $\delta$  in \eqref{cK-delta} specifically precludes the oscillatory instabilities which can arise through 
pulse-pulse interactions, see \cite{DGK-98} and Theorem 1 and Figure 3.1 of \cite{RGG-M} for examples of semi-strong spectra crossing 
into the right-half complex plane as the pulses become too close.  Pulses of this type typically have complex spectra; an examination of 
real eigenvalues is not sufficient for stability.
\end{remark}

The pulse dynamics are driven by the projection of the residual $\cF(\Phi)$ onto the $N$-dimensional space associated to
the broken translational eigenvalues clustered within $O(\eps^r)$ of the origin, where $\eps^r$ measures the
strength of the tail-tail interaction of the {\sl localized pulse component}, and the value of $r$ can be taken arbitrarily large 
by increasing $\ell$ in $\cK_\ell$.  The spectral projection onto the translational eigenvalues,  $\pi_{\vp}$ (see \eqref{pi-def}), induces the 
 complementary spectral space
\beq \label{Xvp} 
X_{\vp} = \left\{ \vU \,\Bigl| \, \|\vU\,\|_X < \infty, \,\, \pi_{\vp}\,\vU =0 \right\},
\eeq
where the $X$-norm, defined in \eqref{Xnorm}, is {\sl locally} polynomially weighted about the position of each pulse.   Our main result is the adiabatic stability of 
the admissible semi-strong $N$-pulse manifolds  in this norm, so long as the asymptotic damping is not too weak.

\begin{theorem} {\bf Adiabatic stability.}\label{thm:main}
Consider the system (\ref{UV}) with $\theta\neq 0$ satisfying the {\sl normal hyperbolicity condition} $\alpha\in[0,\frac12)$.  
Let $\cK$ be an admissible domain of 
pulse configurations and fix $\nu\in[0,\mu)$. Then the associated $N$-pulse manifold $\cM$ is {\sl adiabatically stable} in the norm $\|\cdot \|_X$, 
up to $O\!\left(\eps^{1-\alpha}\right)$. That
is there exists $M_0, d_0, \eps_0>0$ 
such that for all $\eps<\eps_0$  and any initial data $\vU_0:=(U_0, V_0)^T$ of the form
\beq\label{e:init-cond}
 \vU_0 = \Phi(x;\vp)+W_0(x),
 \eeq
with $\dist(\vp,\partial \cK)>d_0$ and $\|W_0\|\leq M_0\eps^\alpha|\ln \eps|^{-1},$ 
the solution of  \eqref{UV} can be uniquely decomposed for $t\in [0, T_b]$ as
\beq
\vec{U}(x,t)= \Phi(x;\vp(t))+W(x,t),
\eeq
where $\vp\in\cK$ is a smooth function of $t$, the remainder $W(\cdot,t) \in X_{\vp(t)}$ satisfies
\beq
\label{Main-bound} \|W(t)\|_X \leq M \left(e^{-\eps^{\alpha}{\nu t}}\|W_0\|_X + \eps^{1-\alpha}\right),
\eeq
and the time $T_b$ for the pulse configuration $\vp=\vp(t)$ to hit $\partial\cK$ satisfies $T_b\geq M_0 d_0\eps^{-1}.$ 
 Moreover, during this time interval the pulse configuration evolves at leading order according to
\beq
\label{Main-pulse} 
\dfrac{\partial \vec p}{\partial t} =\eps Q^{-1}\cA(\vp\,) \vq\,^{\theta} +O\!\left(\epsilon^{2-\alpha}, \eps\|W(t)\|_X, \|W(t)\|_X^2\right),
\eeq
where $Q(\vp\,)=\diag(\vq\,(\vp\,))$ is the diagonal matrix of amplitudes, and the anti-symmetric matrix $\cA$ is given by
\beq
\label{matrixA} 
[\cA]_{kj} = \frac{\alpha_{21}}{\alpha_{22}+1}\frac{ \overline{\phi_0^{\alpha_{12}}}\,\overline{\phi_0^{\alpha_{22}+1}}}{\|\phi_0^\prime\|^2_{L^2}} 
         \left\{ \ba {cc}  e^{-\epsilon^{1+\alpha/2}\sqrt{\mu}|p_k-p_j|} & k>j, \\ 
              0 & k=j, \\ 
          -e^{-\epsilon^{1+\alpha/2}\sqrt{\mu}|p_k-p_j|} & k<j,  \ea \right.
\eeq
in terms of $\vp$ and the pulse profile $\phi_0$ of \eqref{phi0-def} through the total mass of its powers, see \eqref{mass-notation}.
 \end{theorem}
\begin{remark}
Since the matrix $\cA$ has negative entries above the diagonal, and positive below, the pulses generically repel each other, with the 
right- and left-most pulses move right and left respectively. Moreover, if the pulse configuration enters the weak regime, then
then the inter-pulse separations are monotonically increasing and the time $T_b$  is infinite.
\end{remark}
\begin{remark}
The error from the $\|W\|_X^2$ term may initially dominate the pulse motion if $\|W_0\|_X$ is initially $O(\eps^\alpha).$ However, 
after a $t\approx O(\eps^{-\alpha})$ transient, during which the pulses move a negligible $O(\eps^\alpha)$ distance, this term becomes higher order.
After this transient the error bound on the pulse evolution \eqref{Main-pulse} is sharp when the pulses are uniformly in the semi-strong regime, 
that is each pulse separation satisfies $a_0 \eps^{1+\alpha/2}\leq \Delta p_i \leq a_1 \eps^{1+\alpha/2}.$ However, if one or more pulse separations
become large, and some pulses enter into a weak regime, then the leading order pulse evolution is still given by the first term on the
right-hand side of \eqref{Main-pulse}. A proof of this requires the decomposition \eqref{decompU,V} to be built upon a linearization
about the more accurate pulse ansatz, $\Phi+\Phi_c$. Thus we must analyze the linearization about $\Phi$ to construct $\Phi_c$,
see \eqref{vecPhi1}, and then linearize about 
$\Phi+\Phi_c$ to perform the RG iteration. For brevity of presentation, we have forgone these technicalities. 
\end{remark}

The paper is organized as follows. In section 2 we present the existence of the semi-strong $N$-pulses, whose graphs form the semi-strong $N$-pulse
manifolds.  In particular, we uniformly bound the components of $\vq$ both from above, and away from zero from below. 
In section 3 we develop estimates on the resolvent of a key linear operator and on the residual, obtained by evaluating the right-hand side of (\ref{UV}) at a semi-strong $N$-pulse $\Phi$.  In section 4 we characterize the spectrum of the linearization $L_\vp^{(j)}$ about an admissible $N$-pulse, and in section 5 we develop resolvent and semi-group estimates on the full operator.  In section 6 we use the renormalization group approach to develop nonlinear estimates on the full system and obtain the adiabatic stability results. Section 7 presents technical estimates used in section 6.
We conclude with a discussion which motivates a more general relation between the normal hyperbolicity and adiabatic stability.
 
\subsection{Notation}
We fix the number $N\in\mathbb{N}_+$ of localized pulses and the minimal pulse separation parameter, $\ell>0$,
defined in \eqref{cKell-def}. 
For each fixed pulse position $\vp\in\cK$ 
we define the following norm,
\be
\|f\|_{W^{1,1}_{\xi}} = \|\xi f\|_{L^1}+\|\partial_x f\|_{L^1},
\ee
where $\xi$ is a smooth, positive, mass one function with support within $(-\ell/4, \ell/4)$. We also define its $\vp$ translates
\be
\label{xi-def} \xi_j :=\xi(x-p_j).
\ee
The  ${W^{1,1}_{\xi}}$ norm controls $L^{\infty}$, since for any $x, y \in \bR$,
\be
|u(x)| \leq |u(y)|+\int_\bR|u'(z)|dz,
\ee
and multiplying this inequality by the mass-one function $\xi(y)$, and integrating over $y$ yields
\be
|u(x)| \leq \|u'\|_{L^1}+\int_\bR|\xi(y)u(y)|dy = \|u\|_{W^{1,1}_{\xi}}.
\ee
For each $\beta>0$ and pulse configuration $\vp\in\cK$ 
we introduce the locally weighted space $L^1_{\beta,\vp}$, defined through the partition of unity 
\beq\label{chi-def}
\left\{\chi_j(\vp\,)\right\}_{j=1}^N,
\eeq
which is subordinate to the cover $\{ (s_{j-1}+1,s_j-1)\bigl| j=1, \ldots N \}$ with 
$s_j= \frac{p_{j+1}-p_j}{2}$, for $j=\{1,\dots,N-1\}$ while $s_0=-\infty$ and $s_N=\infty$, see Figure\,\ref{fig:pou}.
The space is defined through the corresponding norm
\be
\label{windownorm} \|f\|_{L^1_{\beta,\vp}} = \sum_{j=1}^N\left\| \left(1+|x-p_j|^\beta\right)\chi_j f\right\|_{L^1},
\ee
which controls long-wavelength terms, \textit{uniformly}  in each $\chi_j$ window.  In particular, we have the estimate
\beq
 \left\| \widehat{\chi_jf}\right\|_{W^{1,\infty}} \leq C \|f\|_{L^1_{1,\vp}},
 \eeq
 where the hat denotes the Fourier transform and $W^{1,\infty}$ is the classical Sobolev norm.
\begin{figure}[ht]
\begin{center}
\includegraphics[scale=.6]{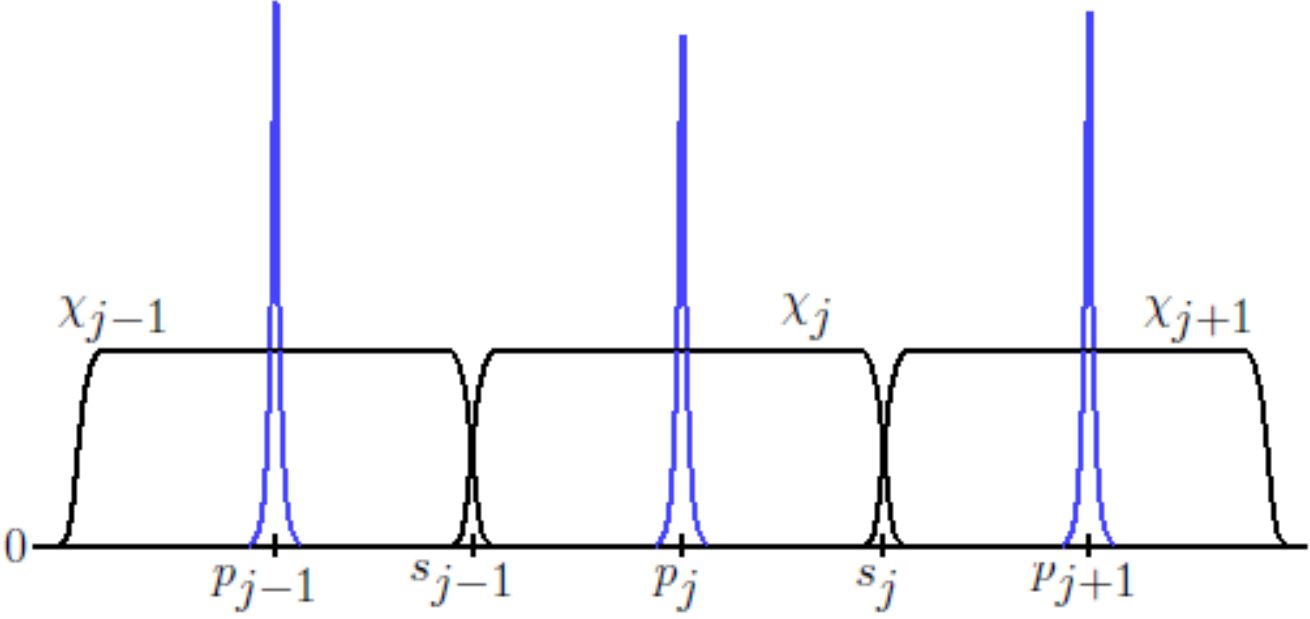}
\caption{The localized component of the pulse solution associated partition of unity $\{\chi_j\}$.} 
\label{fig:pou}
\end{center}
\end{figure}
We define the windowing $\{ f_j=\chi_j f\}_{j=1}^{N}$ of a function $f$ with respect to the pulse configuration $\vp$. This affords the
decomposition
\be
\label{windowing} f=\sum_{j=1}^Nf_j=\sum_{j=1}^Nf\chi_j.
\ee
The main results are stated over the Banach space $X$ defined by the norm
\be
\label{Xnorm} \|F\|_X  = \|f_1\|_{W_{\xi}^{1,1}}+\|f_2\|_{H^1_{\oalpha,\vp}},
\ee
for $F=(f_1,f_2)^T$ where $\oalpha=\max\{\alpha_{11},\alpha_{21}\},$
and
\beq
 \|f\|_{H^1_{\oalpha,\vp}}^2:= \|f^\prime\|_{L^2}^2 + \|f\|_{L^1_{\oalpha,\vp}}^2. 
 \eeq
 The $H^1_{\oalpha,1}$ norm controls the usual $H^1$ norm and affords the nonlinear estimate \eqref{mega-bound} required
 to control the nonlinearity $\cNon$, as in \eqref{nonlinestimate}, for the case $\alpha_{12}=2.$

For  $G \in L^2(\bR)^{k\times l}$ with components $[G]_{ij}=g_{ij}\in L^2(\bR)$,  we denote the tensor operator $\otimes\,G:L^2(\bR)\mapsto \bR^{k\times l}$, 
which acts on $h \in L^2(\bR)$ by component-wise inner product,
\be
 \left[\otimes\,G \cdot h\right]_{ij} =  \left( g_{ij}, h \right)_{L^2}.
\ee
If $F\in L^2(\bR)^{j\times k}$ then the tensor product $F\otimes G$ is  a finite rank map that takes $h\in L^2(\bR)$ to $F\otimes G\cdot h \in L^2(\bR)^{j\times m}$ through the usual matrix 
multiplication of $F$ with  $\otimes\,G\cdot h.$ In particular, for each $\vp\in\cK$ we define the associated windowing tensor
\beq
\otimes\, \vchi^{\,T} = \otimes\,\left( \chi_1, \cdots,  \chi_N  \right),
\eeq
where $\{\chi_j\}_{j=1}^N$ is the partition of unity associated to $\vp$ and the superscript $T$ denotes transposition. 
The windowing tensor reproduces the masses
of the windowings of $f$, in vector form,
\[ \otimes\, \vchi ^{\,T}\cdot f = \left(\, \overline{ f}_1, \cdots, \overline{f}_N\,\right),\]
where we denote the mass of $h\in L^1(\bR)$ by
\beq
\overline{h} := \int_\bR h(x)\, dx.\label{mass-notation}
\eeq
Combining the windowing tensor with the vector function
\beq
\vxi\,^T = \left( \xi_1,  \cdots,  \xi_N  \right),
\eeq
where the $\{\xi_j\}_{j=1}^N$ are defined in (\ref{xi-def}), yields the windowing tensor product $\vxi\otimes\vchi$ which is a rank $N$ map from $L^2(\bR)$ to $L^2(\bR)$ 
defined by
\beq
\vxi^{\,T}\otimes \vchi \cdot f = \sum\limits_{j=1}^N  (\chi_j,f)_2 \xi_j(x) = \sum\limits_{j=1}^N \overline{f}_j \,\xi_j(x).
\eeq
The windowing tensor product replaces $f$ with a sum of $N$ compactly supported functions, centered at the pulse positions $\vp$ with the same mass in each window
as $f$. Indeed, we have the equality 
\[ \otimes\vchi^{\,T}\cdot\left(f - \vxi\otimes \vchi^{\,T} \cdot f\right)=0.\]
For vectors $\vq\in\bR^N$  and $\beta\in\bR$ we introduce the component-wise exponential,
\beq\label{matlab-exp}
\vq\,^\beta := (|q_1|^\beta, \cdots, |q_N|^\beta)^T.
\eeq

\section{Construction of the Semi-strong $N$-Pulse Manifolds}
The localized component of the pulse, $\phi_0$, see \eqref{phi0-def}, decays exponentially at an $O(1)$ rate.
Subsequently we take the pulse-pulse separation distance $\ell>0$ in $\cK_\ell$, recall \eqref{cKell-def}, so large that the localized tail-tail interaction
$ \phi^2((p_j-p_{j+1})/2)\sim \phi^2(\ell/2)\sim O(\eps^r)$ for some $r=r(\ell)\sim \ln\ell \geq 2.$

 To each vector of amplitudes  $\vq\in\bR_+^N$ we will associate a semi-strong $N$-pulse configuration 
 $\Phi=\left(\Phi_1(x;\vp,\vq\,),\Phi_2(x;\vp,\vq\,)\right)^T$, see Figure\,\ref{f:N-pulse}. Moreover we slave
the vector $\vq$ to the positions $\vp$ of the localized pulses via the {\sl mean-field equations}
\be
\label{mf-eq} q_{j} =  \Phi_1(p_j),
\ee
for $j=1,\dots,N$.  The construction begins with the localized fast-pulse $\phi_j$, the unique solution of the equilibrium equation
\be
\label{phij-def} \phi_j^{\prime\prime}-\phi_j+q_j^{\alpha_{21}}\phi_j^{\alpha_{22}}=0,
\ee
which is homoclinic to the origin and symmetric about $x=p_j.$
The existence of $\phi_j$ for $\alpha_{22}>1$ is immediate as the equation has a first integral.  A simple re-scaling shows that 
$\phi_j$ can be written in the from
\be
\label{phij-phi0}
  \phi_j(x;p_j,q_j) = q_j^{-\alpha_{21}/(\alpha_{22}-1)}\phi_0(x-p_j),
\ee
where $\phi_0$ is the unique homoclinic solution to
\be
 \label{phi0-def} 0 =\phi_0^{\prime\prime}-\phi_0+\phi_0^{\alpha_{22}}\!\!,
\ee
which is even about $x=0$. The second component of $\Phi$ is the
sum of the localized pulses,
\be
 \Phi_2= \sum\limits_{j=1}^N \phi_j(x;\vp,\vq\,).
 \label{Phi2-def}
\ee
The first component of $\Phi$ is the mean field generated by the localized pulses, obtained by 
solving the first equation of \eqref{UV} at equilibrium with $U$ replaced in the nonlinearity by the
local constant value $q_j$ when acting on $\phi_j$. The result is a linear equation for $\Phi_1$,
\be
\label{Phi1-def}  
\epsilon^{-2}\partial_{xx} \Phi_1- \epsilon^{\alpha}\mu \Phi_1-\epsilon^{-1}\sum_{j=1}^N q_j^{\alpha_{11}}\phi_j^{\alpha_{12}}+\eps^{\alpha/2}\rho=0.
\ee
Introducing the operator associated to the essential spectrum,
\be
\label{L11e-def} L_{11}^e=-\epsilon^{-2}\partial_{x}^2+ \epsilon^{\alpha}\mu,
\ee
with inverse denoted $L_{11}^{-e}$, we rewrite \eqref{Phi1-def} as
\be \label{Phi1-defb}  
L_{11}^e\left(\Phi_1-\eps^{-\alpha/2}\frac{\rho}{\mu}\right)=-\epsilon^{-1}\sum_{j=1}^N q_j^{\alpha_{11}}\phi_j^{\alpha_{12}},
\ee
and define the $N$-pulse configurations
\be
\label{Phi-def} 
\Phi(x,\vec{p},\vec{q}\,) := \left[ \begin{array} c \Phi_1(x,\vec{p},\vec{q}\,) \\
\Phi_2(x,\vec{p},\vec{q}\,) \end{array} \right]  =  
\left[ \begin{array}{ c } \eps^{-\alpha/2}\frac{\rho}{\mu} - \epsilon ^{-1}L_{11}^{-e}\left(\sum_{j=1}^N q_{j}^{\alpha_{11}}\phi_j^{\alpha_{12}}(x)\right) \\                                                                \sum_{j=1}^N \phi_j(x) \end{array} \right].
\ee

For $\lambda\in\bC$ we introduce the Green's function $G_{\lambda}(x)$ associated to $L_{11}^e+\lambda$ which enjoys the property
\be
\label{greenL11} \left(L_{11}^e+\lambda\right)^{-1}f =\left(G_{\lambda} * f\right)(x).
\ee
From the Fourier transformation, we determine its explicit formula
\be
\label{Ggreen} G_{\lambda}(x) := \dfrac{\epsilon^{2}}{k_{\lambda}}e^{-k_{\lambda}|x|},
\ee
where $k_\lambda$ was introduced in Definition\,\ref{def-one}.
A central role is played by the scaled two-point correlation matrix $\oG_N(\lambda)$ of the Green's function \eqref{Ggreen}
\be
\label{Gtwopointmat} \left[\oG_N(\vp,\lambda)\right]_{ij}:= e^{-k_{\lambda}|p_i-p_j|},
\ee 
where $\vp\in\cK$ for some minimal pulse separation $\ell>0.$ The unscaled version of the two-point correlation matrix is denoted
\beq \label{unscaled-G}
G_N(\vp,\lambda) =  \dfrac{\epsilon^{2}}{k_{\lambda}} \oG_N(\vp,\lambda).\eeq

The following theorem shows the existence of solutions to the mean-field equation. 

\begin{theorem} (\textbf{Existence of $N$-pulse configurations})\label{thm:mean_field}
Fix the pulse separation $\ell>0$ and  $N\in{\mathbb N_+}$. 
There exists a unique function $\vq=\vq(\vp\,)$ for which $\Phi\left(x,\vp,\vq\,(\vp\,)\right)$ solves the mean-field
equation \eqref{mf-eq}. Moreover, the solution takes the form $\vq=\vq_0(\vp\,)+\eps^{1+\alpha/2}\vq_1(\vp\,)$,
and is uniformly bounded, component-wise from above and away from zero.
In particular, $\vq_0$ solves
\beq\label{q0-eq}
 \frac{\overline{\phi_0^{\alpha_{12}}}}{\sqrt{\mu}}\oG_N(\vp,0) \tq_0^{\,\theta} = \frac{\rho}{\mu}\vec{1}-\eps^{\alpha/2}\vq_0.
 \eeq
 and $\vq$ admits the expansion,
\beq\label{q0-exp}
\vq\,(\vp\,)= \left(\frac{\rho}{\sqrt{\mu}\overline{\phi_0^{\alpha_{12}}}} \oG_N^{-1}(\vp,0) \vec{1}\right)^{\dfrac{1}{\theta}}+O\left(\eps^{\alpha/2}\right),
\eeq
where the exponential is component-wise with $\theta$  defined in (\ref{theta-def}).
\end{theorem}

\begin{proof}
An application of \eqref{greenL11} with $\lambda=0$ to the first component of \eqref{Phi-def} evaluated at $x=p_k$
yields the equation
\be
\label{q_kequ} q_k=\eps^{-\alpha/2}\frac{\rho}{\mu}- \epsilon^{-1}\left(G_0 * \sum_{j=1}^N q_{j}^{\alpha_{11}}{\phi_j}^{\alpha_{12}}\right)(p_k),
\ee
where, for $\lambda=0$, the Green's function takes the form
\be
\label{Ggreensfunc} G_{0}(x)= \frac{\epsilon^{1-\alpha/2}}{\sqrt{\mu}}e^{-\epsilon^{1+\alpha/2}\sqrt{\mu}|x|}.
\ee
The system of equations in \eqref{q_kequ} may be written in the vector form
\be
\label{vecq} \vec{q}= \eps^{-\alpha/2}\dfrac{\rho}{\mu}\vec{1}-\eps^{-\alpha/2}\mcG(\vp,\eps) \vec q \hspace{.025in} ^{\alpha_{11}},
\ee
where $\mcG \in \mathbb{R}^{N \times N}$ has entries
\be
\mcG_{jk}= \epsilon^{\alpha/2-1}\left(G_0 * \phi_j^{\alpha_{12}}\right)(p_k).
\ee
Substituting for $G_0$ from (\ref{Ggreensfunc}), for $\phi_j$ from (\ref{phij-phi0}), and recalling the definition (\ref{theta-def}) of $\theta$, we obtain
\be
\label {q_k11} 
\mcG_{jk}= \dfrac{1}{\sqrt{\mu}}q_{j}^{\theta-\alpha_{11}} \int \limits_\bR e^{-\epsilon^{1+\alpha/2}\sqrt{\mu}|y-p_k|} \phi_0^{\alpha_{12}}(y-p_j)dy,
\ee
from which we see that $\mcG$ has an $O(1)$ limit as $\eps$ tends to zero. For $j\neq k$ the exponential decay of $\phi_0$ permits us to Taylor expand the exponential about $y=p_j$,
\be
e^{-\epsilon^{1+\frac{\alpha}{2}}\sqrt{\mu}|p_k-y|} = e^{-\epsilon^{1+\frac{\alpha}{2}}\sqrt{\mu}|p_k-p_j|}-\sigma_{jk} \epsilon^{1+\frac{\alpha}{2}}\sqrt{\mu}e^{-\epsilon^{1+\frac{\alpha}{2}}\sqrt{\mu}|p_k-p_j|}(y-p_j)+O(\epsilon^{2+\alpha}|y-p_j|^2),
\ee
where $\sigma_{jk}={\rm sign}(p_k-p_j).$ Substituting this into  \eqref{q_k11} and evaluating the leading-order integral, we have the expansion
\be
\label{G2} \mcG_{jk}= \dfrac{1}{\sqrt{\mu}}q_{j}^{\theta-\alpha_{11}} \overline{\phi_0^{\alpha_{12}}}e^{-\epsilon^{1+\alpha/2}\sqrt{\mu}|p_k-p_j|}
+O(\eps^{2+\alpha}),
\ee
where the $O(\eps^{1+\alpha/2})$ terms are zero due to parity, and we are applying the total mass notation $\overline{f}:=\int_\bR f(s)ds.$ For the case $j=k$ we consider the Taylor expansion 
of the exponential for $y<p_j$ and for $y>p_j$. The $O(y-p_j)$ terms do not integrate to zero and we obtain
\be
\label{G3} \mcG_{jj}= \dfrac{1}{\sqrt{\mu}}q_{j}^{\theta-\alpha_{11}} \overline{\phi_0^{\alpha_{12}}}
+O(\eps^{1+\alpha/2}).
\ee
Therefore, using both \eqref{G2} and \eqref{G3}, we see that the equation (\ref{vecq}) can be written as
\be
\label{tq-eq} \eps^{\alpha/2}\tq=  -\left( \frac{\overline{\phi_0^{\alpha_{12}}}}{\sqrt{\mu}} \oG_N\,(\vp,0) +\eps^{1+\alpha/2}\mcG_1(\vp,\eps)\right) \tq \, ^{\theta}+\dfrac{\rho}{\mu}\vec{1},
\ee
where  $\mcG_1\in\bR^{N\times N}$ is uniformly bounded  for $|\eps|<\eps_0$ and smooth in $\vp$.  
Rescaling $\tp=\eps^{1+\alpha/2} \vp$, removes the $\eps$ dependence of the exponentials. The existence of the solution to \eqref{q0-eq} of the form \eqref{q0-exp} follows from the implicit function theorem and the invertibility of $\oG_N(\vp,\lambda)$,
indeed 
\beq\label{GN-det}
  \det  \oG_N(\vp,\lambda) = \prod\limits_{i=1}^{N-1} \left(1- e^{-2k_\lambda(p_{i+1}-p_i)}\right), 
  \eeq
  which is non-zero. Moreover, the inverse of $\oG_N$ is tri-diagonal and can be constructed explicitly, see \cite{ironward} for details.
  The positivity and uniform bounds on $\vq$ for $\vp$ in the semi-strong regime follow.  The existence of the solution $\vq$ to \eqref{tq-eq}
then follows from a perturbation off of $\vq_0.$\hfill \end{proof}

\section{Bounds and Residual Estimates on Semi-Strong $N$-pulses}
Fix a set $\cK$ of semi-strong $N$ pulse solutions as constructed in Theorem\,\ref{thm:mean_field}, with
minimal pulse separation distance $\ell>0$. For all $\vp\in\cK$, in subsection 3.1 we establish estimates on $(L_{11}^e+\lambda)^{-1}f$ in various norms, in subsection 3.2 we obtain bounds on the norms of the semi-strong $N$-pulses $\Phi(\vp)$, in subsection 3.3 we establish a result allowing for subsequent reduction of a finite rank operator, and in subsection 3.4 we establish
estimates on the residual $\cF(\Phi(\vp\,)).$  

\subsection{Linear estimates}
We recall  $L_{11}^e$ and $k_{\lambda}$ introduced in  (\ref{L11e-def}) and Definition\,\ref{def-one} respectively.
For each $\vp\in\cK$, we have the partition of unity $\{\chi_j\}_{j=1}^N$ given in (\ref{chi-def}) and the
weighted-windowed norm $L^1_{1,\vp}$ defined in \eqref{windownorm}.
\begin{lemma}
\label{l:L11}
There exists $C>0$ such that for any $f \in L^1(\mathbb{R})$ or $f \in {W^{1,1}_{\xi}}(\mathbb{R})$ and for any $\lambda \in \mathbb{C}$  \verb \   $(-\infty,-\epsilon^{\alpha}\mu)$, the following estimates hold:
\begin{eqnarray}
\label{1lemma1} \hskip .35in \|(L_{11}^e+\lambda)^{-1}f\|_{W^{1,1}_{\xi}} &\leq& \dfrac{C\epsilon^2}{\Re(k_{\lambda})}\min\left\{\|f\|_{L^1}, \dfrac{\|f\|_{W^{1,1}_{\xi}}}{|k_{\lambda}|}\right\}, \\
\label{3lemma1} \hskip .35in \|(L_{11}^e+\lambda)^{-1}f\|_{L^{\infty}} +|\Re(k_{\lambda})|\, \|(L_{11}^e+\lambda)^{-1}f\|_{L^1}&\leq& C\dfrac{\epsilon^2}{|k_{\lambda}|}\|f\|_{L^1}, \\
\label{4lemma1} \hskip .35in \|\partial_x((L_{11}^e+\lambda)^{-1}f)\|_{L^{\infty}} +
| \Re(k_{\lambda})|\, \|\partial_x((L_{11}^e+\lambda)^{-1}f)\|_{L^{1}} &\leq& C\epsilon^2\|f\|_{L^1},
\end{eqnarray}
where $\Re$ denotes the real part.
Moreover, for all $\vp\in\cK$,  $f \in L^1_{1,\vec{p}}(\mathbb{R}),$ and
$\lambda \in \mathbb{C}$  \verb \   $(-\infty,-\epsilon^{\alpha}\mu)$  we have the small-mass estimates
\begin{eqnarray}
\label{5lemma1} \|(L_{11}^e+\lambda)^{-1}f\|_{W^{1,1}_{\xi}} &\leq& C\dfrac{\epsilon^2}{\Re(k_{\lambda})}\left(|\otimes \vec{\chi} \cdot f| +\big{|}k_{\lambda}\big{|}\|f\|_{L^1_{1,\vec{p}}} \right), \\
\label{6lemma1} \|(L_{11}^e+\lambda)^{-1}f\|_{L^{\infty}} &\leq& C\epsilon^2\left( \dfrac{1}{|k_{\lambda}|}|\otimes \vec{\chi} \cdot f| +\|f\|_{L^1_{1,\vec{p}}} \right).
\end{eqnarray}
\end{lemma}
The small-mass estimates are useful in section 6 when we examine the difference of two linear operators, where their difference is large, but the difference has small mass in each window.
\begin{proof}
We introduce $ g(x):= (L_{11}^e+\lambda)^{-1}f = (G_{\lambda} * f)(x)$ where the Green's function $G_{\lambda}$ is given in \eqref{Ggreen}. 
From the identity $g' =G_{\lambda}'*f$ and $L^p$ convolution estimates \cite{Horm}, we have the bounds
\begin{align}
\|g'\|_{L^1}& \leq \|G_{\lambda}'\|_{L^1}\|f\|_{L^1} \leq C\dfrac{\epsilon^2}{\Re(k_{\lambda})}\|f\|_{L^1}, \\
\|\xi g\|_{L^1} &\leq \|\xi\|_{L^1}\|(G_{\lambda} * f)(x)\|_{L^{\infty}} \leq \|G_{\lambda}\| _{L^{\infty}}\|f\|_{L^1} \leq C\dfrac{\epsilon^2}{|k_{\lambda}|}\|f\|_{L^1},
\end{align}
for $\xi$ defined in \eqref{xi-def}. These two estimates establish the $L^1$ bound in \eqref{1lemma1}.  For the $W^{1,1}_\xi$ bound we first observe that 
\be
\|g'\|_{L^1}=\|(G_{\lambda}* f')(x)\|_{L^1} \leq \|G_{\lambda}\|_{L^{1}}\|f'\|_{L^1} \leq C\dfrac{\epsilon^2}{|k_{\lambda}|\Re(k_{\lambda})}\|f\|_{W^{1,1}_{\xi}}.
\ee
Combining this with the estimate
\be
\|\xi g\|_{L^1} \leq \|\xi\|_{L^1}\|(G_{\lambda} * f)(x)\|_{L^{\infty}} \leq \|G_{\lambda}\| _{L^1}\|f\|_{L^{\infty}} \leq C\dfrac{\epsilon^2}{|k_{\lambda}|\Re(k_{\lambda})}\|f\|_{W^{1,1}_{\xi}}  
\ee
yields \eqref{1lemma1}. 

To establish \eqref{3lemma1} we observe that 
\begin{align}
\|g\|_{L^{\infty}} = &\|(G_{\lambda} * f)(x)\|_{L^{\infty}} \leq \|G_{\lambda}\|_{L^{\infty}}\|f\|_{L^1} \leq C\dfrac{\epsilon^2}{|k_{\lambda}|}\|f\|_{L^1}, \\
\|g\|_{L^1}= &\|(G_{\lambda} * f)(x)\|_{L^1} \leq \|G_{\lambda}\|_{L^1}\|f\|_{L^1} \leq C\dfrac{\epsilon^2}{|k_{\lambda}|\Re(k_{\lambda})}\|f\|_{L^1}.
\end{align}
The bounds \eqref{4lemma1} follow from  similar estimates applied to $g^\prime=G_{\lambda}^\prime * f.$

For the small-mass estimates, we window $f$ through its partition of unity, as in \eqref{windowing}, so that
\be
g_j=G_{\lambda} * f_j
\ee
satisfies $g=\sum_j g_j.$  From the definition \eqref{windownorm} of the windowed norm we see
$\|f\|_{L_{1,\vec{p}}^1}=\sum_j \|f_j\|_{L_{1,j}^1}. $
 We  decompose each $f_j$ into a smooth, localized term and a massless part,
\be
f_j=\bar{f}_j\xi_j + y_j',
\ee
for $y_j \in W^{1,1}(\mathbb R)$ and $\xi_j$ defined in \eqref{xi-def}.  Clearly for any $f$, $\|f\|_{L^1} \leq \|f\|_{L_{1,j}^1}$.  Next, we examine
\begin{align}
\|y_j\|_{L^1} = \int_{\mathbb R} \left[\partial_x(x-p_j)\right]|y_j|dx \leq& \int_{\mathbb R} |(x-p_j)y_j'|dx 
= C\int_{\mathbb R} |(x-p_j)\left(f_j-\bar{f}_j\xi_j\right)|dx, \\
\leq&C\left(\|f_j\|_{L_{1,j}^1}+\|f\|_{L^1}\int_{\mathbb R} |(x-p_j)\xi_j|dx\right)
\leq C\|f_j\|_{L_{1,j}^1}.
\end{align}
We decompose $g_j=g_{j,1}+g_{j,0}$ where $g_{j,1}=\bar{f}_jG_{\lambda} *\xi_j$ and $g_{j,0}=G_{\lambda} *y_j'=G_{\lambda}' *y_j$.  Estimating $g_{j,1}$ using \eqref{greenL11} and \eqref{1lemma1}, we have
\be
\|g_{j,1}\|_{W^{1,1}_{\xi}} = \bar{f}_j\|G_{\lambda} *\xi_j\|_{W^{1,1}_{\xi}} \leq C\dfrac{\epsilon^2}{\Re(k_{\lambda})}\bar{f}_j\|\xi_j\|_{L^1} \leq C\dfrac{\epsilon^2}{\Re(k_{\lambda})}\bar{f}_j.
\ee
The function $G_{\lambda}'$ has a jump at $x=0$, so that
$G''_{\lambda} =\left[G''_{\lambda}\right] + \eps^{2}\delta_{0}, $
where  $\left[G''_{\lambda}\right]$ is the point-wise second derivative of $G_{\lambda}$ and $\delta_0$ is the delta function at $x=0.$  This
yields the estimate,
\be
\|g_{j,0}'\|_{L^1} \leq \left(\|\left[G_{\lambda}''\right]\|_{L^1}+\epsilon^{2}\right)\|y_j\|_{L^1} \leq C\dfrac{|k_{\lambda}|}{\Re(k_{\lambda})}\epsilon^2\|f_j\|_{L_{1,j}^1}.
\ee
Using \eqref{4lemma1}, we find
\be
\|\xi g_{j,0}\|_{L^1} \leq \|\xi\|_{L^1}\|G'_{\lambda} *y_j\|_{L^{\infty}} \leq C\|G'_{\lambda}\|_{L^1}\|y_j\|_{L^1}\leq C\epsilon^2\|f_j\|_{L_{1,j}^1}.
\ee  
Summing over $j$, we have \eqref{5lemma1}.  The inequality \eqref{6lemma1} follows using \eqref{4lemma1} and \eqref{3lemma1} respectively.
\hfill\end{proof}
  
\subsection{Bounds on $\Phi(\vp\,)$}
The following lemma establishes bounds on the $N$-pulse solutions over each admissible set.
\begin{lemma}
Let  $\cK$ denote a family of $N$-pulse configurations, as constructed in Theorem \ref{thm:mean_field}.  Then for all $\beta_1,\beta_2>0$
there exists a constant $C >0$ such that $\forall\vp\in\cK$ and $k=1, \cdots N,$
\begin{align}
\label{Phi1Linfty} \|\Phi_1\|_{L^{\infty}}+  \eps^{1+\alpha/2}\left\| \Phi_1-\eps^{-\alpha/2}\frac{\rho}{\mu}\right\|_{L^{1}}+\|\partial_{p_k}\Phi_1\|_{L^1} \leq& C\eps^{-\alpha/2}, \\
\label{dxPhi1Linfty} \|\partial_x\Phi_1\|_{L^{\infty}}+\|\partial_{p_k}\Phi_1\|_{L^{\infty}} 
+ \eps^{-\alpha/2}\left\| \frac{\partial \Phi_2}{\partial p_k}\!\! +\phi_k^\prime\right\|_{H^1}\!\!\!\!+
\left\|\left(\Phi_1^{\beta_1}-\vchi\cdot\vq\,^{\beta_1}\right)\Phi_2^{\beta_2}\right\|_{L^\infty} \leq& C\eps,
\end{align}
where $\vchi=\vchi\,(\cdot,\vp\,)$ is the partition of unity subordinate to $\vp$ defined in \eqref{chi-def}. Moreover, for all $\beta>0$ and all $v\in L^1_{\beta,\vp}$ there exists $C>0$ such that
\beq
           \left\|\left(\Phi_1^\beta-\vchi\cdot \vq\,^\beta\right)v \right\|_{L^1}\leq C\eps \|v\|_{L^1_{\beta,\vp}}.
\label{mega-bound}
\eeq
\end{lemma}
\begin{proof} 
To establish the bounds on the first two terms in \eqref{Phi1Linfty}, we apply \eqref{3lemma1}, with $\lambda=0$, to \eqref{Phi1-defb} and 
recall that the $\vq$ are uniformly $O(1).$ For the final term of \eqref{Phi1Linfty}, we take $\partial_{p_k}$ of \eqref{Phi1-defb} to obtain,
 \beq
  L_{11}^e \partial_{p_k} \Phi_1 = \eps^{-1}\partial_x\left( \sum\limits_j q_j^{\alpha_{11}}\phi_j^{\alpha_{22}}\right)- 
  \eps^{-1}\sum_{ij}\partial_{q_i}\left(q_j^{\alpha_{11}}\phi_j^{\alpha_{22}}\right)\frac{\partial q_i}{\partial p_k}.\eeq
Inverting $L_{11}^e$ and applying \eqref{4lemma1} with $\lambda=0$  to the first term on the right-hand side, and \eqref{3lemma1}
together with the bound on $\partial_{p_k}q_j$ to the second term establishes the result. 
The first estimate in \eqref{dxPhi1Linfty} follows from applying \eqref{4lemma1} to the derivative of \eqref{Phi1-defb} at $\lambda=0$. The smooth dependence of $\vq$ on $\tp$ and the change of variables $\tp=\vp\eps^{1+\alpha/2}$ imply
 that $ |\frac{\partial q_j}{\partial p_k}|=O(\eps^{1+\alpha/2})$, from which the second and third estimates of \eqref{dxPhi1Linfty} follow. The final estimate of \eqref{dxPhi1Linfty} follows from the $O(\eps)$ $L^\infty$-bound on $\partial_x\Phi_1$, the fact that
$\left(\Phi_1^{\beta_1}-\vchi\cdot \vq\,^{\beta_1}\right)(p_j)=0$, and the exponential decay of $\Phi_2$ away from $x=p_j.$   
The estimate \eqref{mega-bound} follows from the bound
$$ \left\|\chi_j\Frac{\Phi_1^\beta-q_j^\beta}{1+|x-p_j|^\beta}\right\|_{L^\infty}\leq C\eps.$$
\hfill\end{proof}

\subsection{Finite-rank reduction lemma}
The following Lemma is a key tool in identifying the finite-rank reduction of the singularly scaled operator $L_{11}^e+\lambda$, by describing its structure.
\begin{lemma} 
For each pulse separation parameter $\ell>0$ there exists $C>0$ such that for all 
$\lambda \in \bC $ \verb \  $(-\infty,-\epsilon^{\alpha}\mu)$ and $\vp \in \cK_{\ell}$, the following holds
\be
\label{residualmatrix} \Big{|}\big{(} (L_{11}^e+\lambda)^{-1}f,g \big{)} _{L^2} - (\otimes \vec{\chi} \cdot f)^T G_N(\vp,\lambda) \otimes \vec{\chi} \cdot g\Big{|} \leq C\epsilon^2\|f\|_{L^1_{1,\vec{p}}}\|g\|_{L^1_{1,\vec{p}}},
\ee
in terms of the unscaled two-point correlation matrix \eqref{unscaled-G}, for all $f,g \in L^1_{1,\vp}$, defined in \eqref{windownorm}.
\end{lemma}

Proof: Windowing $f$ and $g$ as in \eqref{windowing}, we find
\begin{align} 
\label{windowfg} \left((L_{11}^e+\lambda)^{-1}f,g \right)_{L^2} =& \sum_{i,j=1}^N \left((L_{11}^e+\lambda)^{-1}f_i,g_j \right)_{L^2} 
= \sum_{i,j=1}^N \int \int G_{\lambda}(y-x)f_i(y)g_j(x)dy\,dx.
\end{align}
For any $\vp\in\cK,$ the Greens function $G_\lambda$  admits the Taylor expansion about $p_i-p_j$,
\begin{align}
\label{Gtaylor} G_{\lambda}(y-x)
=& G_{\lambda}(p_i-p_j)+G_{\lambda}'(s)\big{(}(y-p_i)-(x-p_j)\big{)},
\end{align}
for some $s=s(x,y) \in \mathbb{R}$. Substituting the Taylor expansion for $G_\lambda$ into \eqref{windowfg} we observe that
\beq
\sum_{i,j=1}^N \int\int G_{\lambda}(p_i-p_j)f_i(y)g_j(x)dy\,dx =(\otimes \vec{\chi} \cdot f)^T G^N\otimes \vec{\chi} \cdot g,
\eeq
while  $\|G_\lambda^\prime\|_{L^\infty} \leq C\eps^2$. Finally, recalling the definition \eqref{windownorm} of the
windowed norm, we obtain \eqref{residualmatrix}. \hfill $\Box$

\subsection{Residual estimates}
Recalling that $\cF$ denotes the right-hand side of \eqref{UV}, then the residual $\cF(\Phi)$ takes the form
\be
\label{resphi} \cF(\Phi) = \begin{pmatrix} \cF_1(\Phi) \\ \cF_2(\Phi) \end{pmatrix} = 
\left(\ba c  \epsilon^{-2}\partial_{x}^2\Phi_1- \epsilon^{\alpha}\mu \Phi_1-\epsilon ^{-1}\Phi_1^{\alpha_{11}}\Phi_2^{\alpha_{12}}+
\eps^{\alpha/2}\rho \\ \partial_{x}^2\Phi_2-\Phi_2+\Phi_1^{\alpha_{21}}\Phi_2^{\alpha_{22}} \ea \right),
\ee
and enjoys the following properties.

\begin{proposition} Let $\cK$ be a family of $N$-pulse configurations, as constructed in Section 2, with minimal pulse separation $\ell$. 
Then there exists $r=r(\ell)$, which grows at an $O(1)$ exponential rate in $\ell$, such that for all $\vp \in \cK$,
 the residual satisfies the following asymptotic formula
\be
\label{resasymp} \left(\ba{c} \cF_1(\Phi) \\ \cF_2(\Phi) \ea \right)=\left(\ba c \epsilon^{-1}\sum_{j=1}^N(\Phi_1^{\alpha_{11}}-q_j^{\alpha_{11}})\phi_j^{\alpha_{12}}+O(\epsilon^r) \\ \sum_{j=1}^N\left(\Phi_1^{\alpha_{21}}-q_j^{\alpha_{21}}\right)\phi_j^{\alpha_{22}} +O(\epsilon^r) \ea \right),
\ee
for $r=r(\ell)>0$ large.  Moreover, there exists $C>0,$ independent of $\epsilon$ and $\vp \in \mathcal K_{\ell}$ such that the following estimate holds,
\beq
\label{3lemma2} \eps \left(\|\cF_1(\Phi)\|_{L^1}+\|\nabla_\vp\cF_1(\Phi)\|_{L^1}\right)+ \|\cF_2(\Phi)\|_{L^2} +\|\nabla_\vp\cF_2(\Phi)\|_{L^2}\leq C\eps.
\eeq
\end{proposition}
\begin{proof}
 We first examine $\cF_2(\Phi)$ in the $L^2$ norm.  Using \eqref{phij-def} to eliminate the dominant terms we find
\beq
\label{R_2phij1} \|\cF_2(\Phi)\|_{L^2} \leq 
\left\|\Phi_1^{\alpha_{21}}\left(\left(\Sigma_{j=1}^N\phi_j\right)^{\alpha_{22}}-\Sigma_{j=1}^N\phi_j^{\alpha_{22}}\right)\right\|_{L^2}
+ \left|\left|\Sigma_{j=1}^N\left(\Phi_1^{\alpha_{21}}-q_j^{\alpha_{21}}\right)\phi_j^{\alpha_{22}}\right|\right|_{L^2}.
\eeq
The first term  is dominated by tail-tail interactions between the localized
pulses $\phi_j$. Since these pulses decay at an $O(1)$ rate and since $\alpha_{22}>1$, $\phi_j\phi_k=O(\eps^r)$ for $j\neq k$ in any polynomially weighted norm where
$r=r(\ell)$ grows exponentially in pulse separation $\ell.$ It follows from the $L^\infty$ bound,  \eqref{Phi1Linfty}, on $\Phi_1$ that
this term is $O(\eps^r)$ where $r$ can be made as large as desired.  
It is the second term which is dominant, which from the last estimate of \eqref{dxPhi1Linfty} is $O(\eps)$.

To examine the $L^1$ norm of $\cF_1(\Phi)$ we use \eqref{Phi1-defb} to rewrite the first of component of $\cF(\Phi)$ and apply the triangle inequality, 
\beq
\label{R_1phij1} \left|\left|\cF_1(\Phi)\right|\right|_{L^1} \leq
 \epsilon^{-1}\left|\left|\Sigma_{j=1}^N(\Phi_1^{\alpha_{11}}-q_j^{\alpha_{11}})\phi_j^{\alpha_{12}}\right|\right|_{L^1}
+ \epsilon^{-1}\left|\left|\Phi_1^{\alpha_{11}}\left((\Sigma_{j=1}^N\phi_j)^{\alpha_{12}}-\Sigma_{j=1}^N\phi_j^{\alpha_{12}}\right)\right|\right|_{L^1}.
\eeq
The second term is dominated by the tail-tail interaction and is $O(\eps^r),$ while the first term is $O(1)$ by \eqref{dxPhi1Linfty}.
The asymptotic formula \eqref{resasymp} follows by identifying the leading order terms. 

The estimates on the $\nabla_\vp$ terms
follow from the bounds of \eqref{dxPhi1Linfty} which show that $\|\nabla_\vp\, \Phi_1\|_{L^\infty}$ is of the same order as
$\|(\Phi_1-q_j)\phi_j\|_{L^\infty}.$
\hfill \end{proof}

\section{The Linearization and its Spectrum}
The linearization of  $\cF$ about $\Phi(\cdot, \vp)$ is the linear operator $L_{\vp}$ defined by
\be
\label{lin-op} L_{\vec{p}} \equiv \left( \ba{cc}  -L_{11}^e - \epsilon^{-1}\alpha_{11}\Phi_1^{\alpha_{11}-1}\Phi_2^{\alpha_{12}} & -\epsilon^{-1}\alpha_{12}\Phi_1^{\alpha_{11}}\Phi_2^{\alpha_{12}-1} \\  &  \\ \alpha_{21}\Phi_1^{\alpha_{21}-1}\Phi_2^{\alpha_{22}} & \partial_x^2-1+ \alpha_{22}\Phi_1^{\alpha_{21}}\Phi_2^{\alpha_{22}-1}  \ea \right).
\ee
A direct study of these operators is complicated by the spatially varying potentials. In what follows we show that the singular
nature of the term $L_{11}^e$ allows us to approximate the spatially varying terms with finite-rank operators.  In this vein 
we introduce the reduced linearization
\be
\label{tildeL} \tilde{L}_{\vec{p}} \equiv 
\left( \ba{cc}  -L_{11}^e  & 0 \\  &  \\ J_{21} & \tilde{L}_{22}  \ea \right) - \epsilon^{-1}\left( \ba{cc}  J_{11} & J_{12} \\  &  \\ 0 & 0 \ea \right),
\ee
where $J_{11}$ and $J_{12}$ are finite rank operators
\begin{align}
\label{J11} J_{11} =& \alpha_{11}\vec{\xi}\hspace{.02in}^T \otimes (\Phi_2^{\alpha_{12}}\Phi_1^{\alpha_{11}-1}\vec{\chi}), \\
\label{J12} J_{12} =& \alpha_{12}\vec{\xi}\hspace{.02in}^T \otimes (\Phi_2^{\alpha_{12}-1}\Phi_1^{\alpha_{11}}\vec{\chi}),
\end{align}
and the operators on the second row are 
\begin{align}
\tilde{L}_{22} = & \partial_x^2-1 + \alpha_{22}\sum_{j=1}^N  \phi_0^{\alpha_{22}-1}(x-p_j), \\
\label{J21-def}
J_{21} =  &\alpha_{21}\sum_{j=1}^N q_j^{{\alpha_{21}}-1}\phi_j^{\alpha_{22}}.
\end{align}
The differences between ${L}_{\vec{p}}$ and $\tilde{L}_{\vec{p}}$ are large, but the difference has zero mass in each $\vp$-window and
hence is strongly controlled by the singular structure of $L_{11}^{-e}.$ The essential spectra of $L_{\vec{p}}$ and $\tilde{L}_{\vec{p}}$ coincide,
\be
\sigma_{ess}(L_{\vec{p}})=\sigma_{ess}(\tilde{L}_{\vec{p}})=B:=\left\{-\epsilon^{-2}k^2-\epsilon^{\alpha}\mu | k \in \mathbb{R} \right\}.
\ee
Although we do not pursue this issue, the point spectra of the two operators are also asymptotically close. 

\subsection{The point spectrum}

Modulo the finite-rank perturbations, the operator $\tL_{\vec p}$ is lower triangular. This reduced structure affords a precise
characterization of its point spectrum, up to the analysis of an explicit $N\times N$ matrix. So long as the
pulse-pulse separation distance $\ell>0$ in $\cK_\ell$ renders the local tail-tail interactions higher order, then the point spectrum is controlled by the slowly varying component $U$.  
Estimates on the resolvent of $\tL_{11}$ are given in Lemma\,\ref{l:L11}, the following Lemma gives bounds on the resolvent of $\tL_{22}$.

\begin{lemma}
Fix a contour $\cC$ of type \eqref{contour} and $\ell>0$. Then for all $\beta>0$ there exists $C>0$ such that
\beq
   \| (\tL_{22}-\lambda)^{-1} f\|_{H^1_{\beta,\vp}}\leq C\|f\|_{L^1_{\beta, \vp}}, \label{tL22-est}
 \eeq
 for all $\lambda$ on and to the right of $\cC$ and for all $\vp\in\cK.$
 \end{lemma}
 \begin{proof}
 The Green's function $G_{22}$ for $\tL_{22}-\lambda$ decays at an $O(1)$ exponential rate uniformly for $\lambda$ to the right of $\cC$
 since $\lambda$ is an $O(1)$ distance from the essential spectrum of $\tL_{22}$. The rate is also uniform in $\vp$ since
 the $N$-pulses $\{\phi_j\}_{j=1}^N$ decay exponentially outside $O(1)$ intervals. 
 Introducing $\tilde g=G_{22}*f$, we may decompose it as $\tilde g  =  \sum \tilde g_j$, where $\tilde g_j=G_{22}*\chi_j f$. From classic convolution estimates we have
 \beq \| | (1+|x-p_j|^\beta) \tilde g_j \|_{L^1} \leq C \| G_{22}\|_{L^1} \| (1+|x-p_j|^\beta)\chi_j f \|_{L^1} \leq C \|f\|_{L^1_{\beta,\vp}}. \eeq
 The derivation of the $L^2$ bound on $\tilde g^\prime$ is similar. \hfill \end{proof} 

Using the bounds on the resolvents of the diagonal elements of $\tL_\vp$, we establish the following result.

\begin{proposition} Fix the pulse-pulse separation distance $\ell>0$ so large that the localized tail-tail 
interaction is $O(\eps^r)$ for $r=r(\ell)\geq 2.$ Then there exists $\nu_0 > 0$ such that for all  $\vec p \in \mathcal K,$ 
\be
\sigma_p(\tilde{L}_{\vp}) \cap \{\Re(\lambda) > -\nu_0 \} = \sigma_0(\vp\,) \cup \sigma_{ss}(\vec p\,),
\ee
where $\sigma_0(\vp\,)$ consists of $N$ semi-simple $O(\eps^r)$ eigenvalues which are in $\sigma_p(\tilde{L}_{22})$,
and the semi-strong spectrum is comprised, up to multiplicity, of the solutions to
\be
  \sigma_{ss}(\vp\,) =\left\{\lambda\, \Bigl|\, {\det}(I+\cN_\lambda(\vp\,))=0\right\}, 
 \label{ss-spectrum}
\eeq
where the $N\times N$ matrix $\cN_\lambda(\vp\,)$ is given in \eqref{N}.
The eigenspace associated to $\sigma_0$ is contained, up to $O(\epsilon^r)$, within the space
\be \label{cV-exp}
\cV=span\left\{ (0, \phi_1^\prime)^T, \cdots, ( 0,  \phi_N^\prime)^T \right\}.\\
\ee
\end{proposition}

\begin{proof}
 The eigenvalue problem for $\tL_{\vp}$ takes the form 
\begin{align}
\label{evp1} -(L_{11}^e+\lambda)\Psi_1=& \epsilon^{-1}(J_{11}\Psi_1+J_{12}\Psi_2), \\
\label{evp2} (\tilde{L}_{22}-\lambda)\Psi_2=& -J_{21}\Psi_1.
\end{align}
Assume that $\lambda$ is outside of the essential spectrum, $B$, and the point spectrum of $\tL_{22}$.
Then we may invert $\tilde{L}_{22}-\lambda$ in the second equation and eliminate $\Psi_2$ from the first equation.
Inverting $L_{11}^e+\lambda$, we arrive at the scalar equation,
\be
\Psi_1=- \epsilon^{-1}(L_{11}^e+\lambda)^{-1}\left(J_{11}-J_{12}(\tilde{L}_{22}-\lambda)^{-1}J_{21}\right)\Psi_1.
\ee
Recalling \eqref{J11}-\eqref{J12} we rewrite the right hand side in terms of a single, finite-rank operator,  
\be
\label{Psi1} \Psi_1= -\epsilon^{-1}(L_{11}^e+\lambda)^{-1}J_l^{T} \otimes J_r \cdot \Psi_1,
\ee
where the left and right components of the tensor product are
\begin{align}
\label{J_l} J_l =& \vec{\xi}, \\
\label{J_r} J_r =& (\alpha_{11}\Phi_2^{\alpha_{12}}\Phi_1^{\alpha_{11}-1}-\alpha_{12}\Phi_2^{\alpha_{12}-1}(\tilde{L}_{22}-\lambda)^{-1}J_{21}\Phi_1^{\alpha_{11}})\vec{\chi}.
\end{align}
The eigenfunctions outside of $\sigma_p(\tL_{22})$ reside inside an $N$-dimensional space. To resolve \eqref{Psi1} we act on it  
with $\otimes J_r$
\be
\otimes J_r \cdot \Psi_1 =-\epsilon^{-1}\otimes J_r \cdot (L_{11}^e+\lambda)^{-1}J_l^{T} \otimes J_r \cdot \Psi_1.
\ee
Introducing the matrix
\be
\label{N} \cN_{\lambda}(\vec{p})=\epsilon^{-1}\otimes  J_r \cdot (L_{11}^e+\lambda)^{-1}J_l^{T},
\ee
we see the search for semi-strong point spectrum reduces to solving the matrix equation
\be
(I+\cN_{\lambda})\otimes J_r \cdot \Psi_1=0.
\ee

If  $\lambda\notin\sigma_p(\tL_{22})$ and $I+\cN_{\lambda}$ is invertible, then $\otimes J_r \cdot \Psi_1=0$. This combined with \eqref{Psi1} 
implies $\Psi_1=0$, and from \eqref{evp2}, we see that $\Psi_2=0$.  
Conversely if for some $\vec{v}\in\bC^N$ we have $(I+\cN_{\lambda})\vec v=0$, then
\beq
\Psi:= \eps^{-1} \begin{pmatrix} - (L_{11}^e+\lambda)^{-1}J_l^{T} \vec v \cr
                                                                         (\tL_{22}-\lambda)^{-1} J_{21} (L_{11}^e+\lambda)^{-1}J_l^{T} \vec v\end{pmatrix},
\ee
is an eigenvector of $\tL_{\vp}.$ That is,  $\lambda \in \mathbb C \backslash \left(B \cup \sigma_p(\tilde L_{22}) \right)$ is an eigenvalue 
of $\tilde L_{\vp}$ if and only if $I+\cN_{\lambda}$ is singular. The statement on multiplicity follows by considering perturbations of $\tL_{\vp}$ which
break any non-simple eigenvalues into collections of simple spectra and standard results on continuity of eigenvalues. 

To address the point spectrum of $\tL_{\vp}$ arising from $\tL_{22}$ we first recall the defining equation, \eqref{phi0-def}, for 
the localized pulse $\phi_0$.  It is natural to introduce the linearization of \eqref{phi0-def},
\beq \label{L0-def}
L_0 = \partial_x^2 -1 +\alpha_{22}\phi_0^{\alpha_{22}-1},
\eeq
about $\phi_0$. This Sturm-Liouville operator has a simple kernel, spanned by $\phi_0^\prime$, and a single positive eigenvalue, $\lambda_0$,
corresponding to a non-zero ground state eigenfunction $\psi_0>0$. The remainder of the spectrum of $L_0$ is 
contained within $(-\infty, -2\nu_0]$ for some $\nu_0>0.$  The classical result, \cite{Jones}, states that, up to multiplicity,  
$\sigma_p(\tL_{22})$ consists of $N$ copies of $\sigma_p(L_0)$ shifted by at most $O(\eps^r),$ where recall $r \geq 2$, with eigenfunctions consisting of 
linear combinations of $\vp$ translates of the corresponding eigenfunction of $L_0.$ 

Assume that $\lambda\in\sigma_p(\tL_{22})\backslash B$; if the equation \eqref{evp2} is solvable, then the process follows the 
steps outlined above, and returns us to  the semi-strong eigenvalue condition.
The novelty lies in the possibility that  \eqref{evp2} has a nontrivial solution with $\Psi_1=0,$  which from \eqref{evp1} requires that $J_{12}\Psi_2=0$. A standard implicit function argument shows that this is possible for the set of eigenvalues $\sigma_0:=\{\lambda_{1k}\}_{k=1}^N$ of $\tL_{22}$ clustered near zero (the translational eigenvalue of $L_0$) since the corresponding eigenfunctions are locally odd about each $\vp.$
The corresponding eigenspace is denoted $\cV.$ For the eigenvalues $\{\lambda_{0,k}\}_{k=1}^N$ of $\tL_{22}$ clustered near the ground-state, $\lambda_0>0$, of $L_0$, then
\beq
  \Psi_{2k} = \sum\limits_{j=1}^N b_{kj} \psi_0(x-p_j) +O (\eps^r),
 \eeq
 and using the positivity of $\psi_0$ and $\chi$, the $N\times N$ linear system $J_{12} \Psi_{2k}=0$ yields only the trivial solution for 
 the coefficients $b_{kj}.$ The rest of the spectrum of $L_0$ is to the left of $-2\nu_0$, and so standard results on the spectrum of well-separated pulses, \cite{Jones}, imply that the remainder of the spectrum of $\tL_{22}$ lies on the real axis, to the left of $-\nu_0.$ \hfill\end{proof}

We now find an explicit representation for the matrix $\cN_\lambda$, and we show the semi-strong spectrum is partially characterized through the meromorphic function
\beq
\label{Rlambda} \mathcal{R}(\lambda) := \left((L_0-\lambda)^{-1}\phi_0^{\alpha_{22}}, \phi_0^{\alpha_{12}-1}\right)_{L^2},
\eeq
which is analytic except for poles at some of the eigenvalues of $L_0$. A similar function was introduced in \cites{IndRD, RGG-M}.

\begin{proposition} \label{prop:4} 
The matrix $\cN_\lambda$ takes the form
\beq
\cN_\lambda = \Frac{E(\lambda)}{\sqrt{\lambda+\eps^\alpha\mu}} \oG_N\!(\vp\,,\lambda)\left[Q(\vp\,)\right]^{\theta-1} +O(\eps) \label{Nmatrix}
\eeq
where
\beq
E(\lambda) = \alpha_{11}\overline{\phi_0^{\alpha_{12}}}- \alpha_{12}\alpha_{21}\mathcal{R}(\lambda), \label{E-def}
\eeq
the scaled two-point correlation function $\oG_N$ is defined in \eqref{Gtwopointmat} and $Q=\diag({\vq})$ is the diagonal matrix of amplitudes.
\end{proposition}

\begin{proof} The reduction of the formula \eqref{N} for the semi-strong matrix $
\cN_\lambda$ requires the inversion of two second order operators.
The first inversion, for $L_{11}^e+\lambda$, is accomplished by \eqref{residualmatrix}. The second is the inversion of $(\tL_{22}-\lambda).$
This we achieve via the non-local eigenvalue (NLEP) machinery developed in \cite{GS1}.  
For functions that are exponentially localized around the pulse positions $\vp\in\cK$,
the NLEP analysis inverts $\tL_{22}-\lambda$ in a windowed manner using translates of the operator $L_0$,
\beq
L_{0,k} :=\partial_x^2-1 + \alpha_{22}\phi_0^{\alpha_{22}-1}(x-p_k).
\eeq
In particular, since the potential $J_{21}$ is comprised of $N$-pulses localized about the positions $\vp$, the last estimate of \eqref{dxPhi1Linfty} yields
\be 
(\tL_{22}-\lambda)^{-1}J_{21}\Phi_1^{\alpha_{11}} = \alpha_{21}\sum_{k=1}^N q_k^{\alpha_{11}-1-\frac{\alpha_{21}}{\alpha_{22}-1}}(L_{0,k}-\lambda)^{-1}\phi_0^{\alpha_{22}}(x-p_k)+ (\eps^r),
\ee
where the error is in $L^{1,1}_\vp$ for $r \geq 2$.  Using the same estimate we may rewrite the $k'th$ component of $J_r$, defined in \eqref{J_r}, as
\beq
J_{rk}=
 \sum_{k=1}^Nq_k^{\theta-1}\chi_k \left(\alpha_{11}\phi_0^{\alpha_{12}}(x-p_k)- \alpha_{12}\alpha_{21}\phi_0^{\alpha_{12}-1}(x-p_k)\Xi(x-p_k)\right)+(\eps^r), 
\eeq
where the function $\Xi(\lambda):=(L_0-\lambda)^{-1}\phi_0^{\alpha_{22}}$ decays exponentially at a rate proportional to the distance of 
$\lambda$ to $\sigma_{ess}(\tilde{L}_{22})$.  
Turning to \eqref{N}, we address the $ij$ entry of $\cN_{\lambda}$,
\begin{align}
[\cN_\lambda]_{ij}=\hspace*{0.05in}&\epsilon^{-1}\left((L_{11}^e+\lambda)^{-1}J_{li}, J_{rj}\right)_{L^2}, \nonumber\\
\nonumber =\hspace*{0.05in}&\epsilon^{-1}q_j^{\theta-1}\left( (L_{11}^e+\lambda)^{-1}\xi_i, \alpha_{11}\phi_0^{\alpha_{12}}(x-p_j) \right)_{L^2}-\\
&\epsilon^{-1}q_j^{\theta-1}\left((L_{11}^e+\lambda)^{-1}\xi_i,\alpha_{12}\alpha_{21}\phi_0^{\alpha_{12}-1}(x-p_j)\Xi(x-p_j) \right)_{L^2} +O(\eps). \nonumber
\end{align}
To invert $L_{11}^e+\lambda$ we apply \eqref{residualmatrix} to find 
\begin{align}
 [\cN_\lambda]_{ij}=\hspace*{0.05in} & \epsilon^{-1}q_j^{\theta-1}\otimes \vec{\chi} \cdot \left( \alpha_{11}\phi_0^{\alpha_{12}}(x-p_j)\right) G^N_{ij} \otimes \vec{\chi} \cdot \xi_i - \nonumber\\
&  \epsilon^{-1}q_j^{\theta-1}\otimes \vec{\chi} \cdot \left( \alpha_{12}\alpha_{21}\phi_0^{\alpha_{12}-1}(x-p_j)\Xi(x-p_j)\right) G^N_{ij} \otimes \vec{\chi} \cdot \xi_i+O(\eps). \nonumber
\end{align}
Recalling the scaled, two-point correlation matrix $\oG_N$, from \eqref{Gtwopointmat}, and the meromorphic function $\cR$ introduced
in \eqref{Rlambda}, we may represent $\cN_{\lambda}$ as in \eqref{Nmatrix}.
\hfill \end{proof}

We are now in a position to prove the Admissibility Proposition stated in the Introduction.\\
\noindent{\bf Proof of Admissibility:}
We denote the eigenvalues of $\cN_\lambda$ by $\{\mu_j^{(N)}(\lambda,\vp\,)\}_{j=1}^N.$ The semi-strong spectrum is precisely
the set of $\lambda$ for which $\mu_j^{(N)}(\lambda,\vp\,)=-1$ for some $j=1, \cdots, N.$
For the case of a single pulse, $N=1$, the matrix $\cN_\lambda$ is a scalar and
\beq \label{mu1-def}
\mu^{(1)}_1(\lambda) = \frac{\eps}{k_\lambda}E(\lambda)q_\infty^{\theta-1}+O(\eps),
\eeq
where we have introduced the constant
\vspace {-0.1in}
$$q_\infty= \left(\frac{\rho}{\overline{\phi_0^{\alpha_{12}}}\sqrt{\mu}}\right)^{\dfrac{1}{\theta}},$$
corresponding to the amplitude of a single pulse.  Recalling the operator \eqref{L0-def} and the pulse equation \eqref{phi0-def}, we find the identity
\beq
L_0 \phi_0 = \phi_0''-\phi_0+\phi_0^{\alpha_{22}}+(\alpha_{22}-1) \phi_0^{\alpha_{22}} = (\alpha_{22}-1)\phi_0^{\alpha_{22}},
\eeq
which permits us to evaluate the meromorphic function $\cR$ at $\lambda=0$,
\beq
\cR(0) =\left(L_0^{-1}\phi_0^{\alpha_{22}},\phi_0^{\alpha_{12}-1}\right)_{L^2}= \frac{\overline{\phi_0^{\alpha_{12}}}}{\alpha_{22}-1}.
\eeq
In particular, $E(0)= \theta \overline{\phi_0^{\alpha_{12}}}\neq0$ if $\theta\neq0.$ Moreover for $|\lambda|$ sufficiently far from the essential
spectrum, $\cR(\lambda)$ tends to zero so that $E(\lambda)\to \alpha_{11}\overline{\phi_0^{\alpha_{12}}}$, while $|k_\lambda|$ grows, so that $\mu_1^{(1)}(\lambda)$, given in \eqref{mu1-def}, 
tends to zero in this limit. As a consequence,  the assumption that the one-pulse is spectrally stable implies there exists $\nu\in(0,\mu)$ and 
$b>0$ sufficiently large such that $\mu_1^{(1)}$ is uniformly bounded away from $-1$ on the contour $\cC_{\nu,b}.$ Moreover, since $E(0)\neq0$ and
$E$ varies at an $O(1)$ rate in $\lambda$ away from its poles, there exists $C>0$ such that
\beq\label{one-pulse-est}
   \left(1+\mu_1^{(1)}(\lambda)\right)^{-1} \leq C \left(1+\frac{\eps}{|k_\lambda|}\right)^{-1},
 \eeq 
for all $\lambda \in \cC.$ 

Consider an $N$-pulse and a set $\cK$ of pulse positions satisfying \eqref{cK-delta} for some $\delta>0$. Fix a contour $\cC$
for which \eqref{one-pulse-est} holds and consider  $\lambda$ on and to the right of the contour. If in addition $\lambda$ is
sufficiently far from the branch point $-\eps^\alpha\mu$ so that $|\lambda+\eps^\alpha\mu|>\eps^\gamma$, then 
$\Re (k_\lambda) \Delta p_i\geq C \eps^{\gamma-\alpha}$. Indeed, for these $\lambda$ the definition \eqref{Gtwopointmat} and the formula
\eqref{q0-exp} yield the estimate 
\beq\label{delta-est}
\oG_N(\lambda,\vp) Q^{\theta-1} = q_\infty^{\theta-1} I_{N\times N}+O\!\left(\delta^{-1}, \exp\left[-\delta \eps^{\frac{\gamma-\alpha}{2}}\right]\right),
\eeq
and from \eqref{Nmatrix} we see that the matrix $\cN_\lambda$ approximately diagonal with eigenvalues $\mu^{(N)}_j(\vp,\lambda)=\mu_1^{(1)}(\lambda)+O(\delta^{-1})$
for $j=1, \cdots, N$. It is here that we
require $\delta$ sufficiently large, independent of $\eps$, so that none of the eigenvalues $\mu_j^{(N)}$ attains the value $-1$  for $\lambda$ on this set.  
This restriction on the strength of the semi-strong interaction precludes the point spectrum crossing the imaginary axis 
away from the origin, thereby inducing oscillatory instabilities, as is known to happen \cites{DGK-98, DGK-02} for $N$ pulse configurations with pulses that are
too close. 

On the other hand,
if  $|\lambda+\eps^\alpha\mu|\leq  \eps^\gamma$, then $E(\lambda)$ is close to $E(0)$, and is uniformly bounded away from zero.
Moreover, one can show from an inductive proof that
$$\det \oG_N = \prod\limits_{i=1}^{N-1}\left(1-e^{2k_\lambda\Delta p_i}\right)\geq C\prod\limits_{i=1}^{N-1}\Frac{|k_\lambda|\Delta p_i}{1+|k_\lambda|\Delta p_i}\geq C\delta^N,$$
since $\lambda$ is close to the branch point, and $\Delta p_i \geq \delta \eps^{-(1+\alpha/2)}.$ Similarly, $Q$ is diagonal, with positive entries that are uniformly bounded away from zero and infinity, thus the determinant of $Q^{\theta-1}$ is bounded away from zero. 
Moreover, the eigenvalues of $\oG_N Q^{\theta-1}$ are uniformly bounded from above, while we have shown that their product is 
uniformly bounded from below in modulus; it follows that each eigenvalue is uniformly bounded from below in modulus. 
We combine these facts into the estimate on the eigenvalues of $\cN_\lambda$,
\beq
 |\mu_j^{(N)}|\geq C \frac{\eps}{|k_\lambda|}\gg 1, 
 \eeq
 for $j=1, \ldots, N$ for all $\lambda$ satisfying $|\lambda+\eps^\alpha\mu|\leq  \eps^\gamma$ which are on and to the right of $\cC.$ Combining
 the two types of estimates we see that each $\mu_j^{(N)}$ is uniformly bounded away from $-1$ for all $\lambda$ on and to the right of $\cC$
 and for all $\vp\in\cK$. Moreover, the matrix  $I+\cN_\lambda$ is uniformly invertible on these sets, in particular it satisfies \eqref{Nlam-est},
 which establishes the admissibility of $\cK.$
 
We remark that since $\phi_0$ and $\Psi_0$, the ground state eigenfunction of $L_0$, are both positive functions, it follows that
 $E(\lambda)\to -\infty$ as $\lambda\to\lambda_0$ from the right along the real axis, where $\lambda_0$ is the ground-state
eigenvalue. In particular, if $E(0)>0$ then it must be that $\mu_1^{(1)}(\lambda)$ attains the value $-1$ on $(0,\lambda_0)\subset\bR$ and
hence there is a semi-strong eigenvalue on that segment. Thus the assumption that the one-pulse is linearly stable requires
that $\theta<0.$  \hfill $\Box$

\subsection{The Spectral Projection and Adjoint Eigenfunction Asymptotics}
It is essential to our analysis that we control the projection onto  the $N$-dimensional eigenspace of $\tilde{L}$
associated to the eigenvalues $\sigma_0$ that are algebraically close to the origin.
The spectral projection takes the form
\be
\pi_{\vec{p}}\vec{U} \equiv \sum_{j=1}^N\dfrac{(\vec{U},\Psi_j^{\dagger})}{(\Psi_j,\Psi_j^{\dagger})}\Psi_j,\label{pi-def}
\ee
where $\cV:=\{\Psi_j\}_{j=1}^N$ and $\cV^\dag:=\{\Psi_j^\dag\}_{j=1}^N$ are bi-orthogonal bases for the eigenspaces of $\tilde{L}$
and its adjoint respectively. The complementary projection is given by 
\be
\label{tildepi} \tilde{\pi}_{\vec{p}}\vec{U} \equiv \vec{U}-\pi_{\vec{p}}\vec{U}.
\ee
The basis elements of $\cV$ satisfy the expansion (\ref{cV-exp}); in the lemma below we develop asymptotic expansions for 
the elements of $\cV^\dag.$
 
\begin{lemma}\label{L:Psi-dag} Let $\cK$ be a collection of admissible pulse configurations. 
Then there exists $C>0$, independent of $\eps>0$ and $\vp\in\cK$, such that the basis elements of the adjoint eigenspace $\cV^\dag$ satisfy 
\be
\label{adjointest} \|\Psi_{1,k}^{\dagger}\|_{W^{1,1}_{\xi}}+\epsilon^{1-\alpha/2}\|\Psi_{2,k}^{\dagger}-\phi_k'\|_{H^1} \leq C\eps^2,
\ee
for $k=1, \cdots, N.$
\end{lemma}

Proof:  The adjoint operator is given by:
\be
\tilde{L}^{\dagger} =  \left( \ba{cc}  -L_{11}^e  & J_{21} \\  &  \\ 0 & \tilde{L}_{22}  \ea \right) - \eps^{-1}\left( \ba{cc}  J_{11}^{\dagger} & 0 \\  &  \\ J_{12}^{\dagger} & 0 \ea \right),
\ee
where
\begin{align}
J_{11}^{\dagger}=& \alpha_{11}\vec{\chi}\hspace{.02in}^T \Phi_2^{\alpha_{12}}\Phi_1^{\alpha_{11}-1}\otimes \vec{\xi},\\
J_{12}^{\dagger}=& \alpha_{12}\vec{\chi}\hspace{.02in}^T \Phi_2^{\alpha_{12}-1}\Phi_1^{\alpha_{11}}\otimes \vec{\xi}.
\end{align}
For the algebraically small eigenvalues $\lambda\in \sigma_0$, the operator $\tL_{22}$ is almost singular with
kernel approximately spanned by $\{\phi_k^\prime\}_{k=1}^N$. This permits us to 
normalize the second component of the $k'th$ basis elements as
\beq
\Psi_{2,k}^{\dagger}:= \eps^{-1}\tL_{22}^{-1}J_{12}^{\dagger}\Psi_{1,k}^{\dag}= \phi_k^\prime  +\eps^{-1}\tL_{22}^{-1}\tpi_{22} J_{12}^{\dag}\Psi_{1,k}^{\dag},\label{psi2dag}
\eeq
where $\tpi_{22}$ is the projection off of the small eigenspace of the self-adjoint operator $\tL_{22}.$ Using this form
for $\Psi_{2,k}^\dag$ and proceeding as in Proposition\,\ref{prop:4},  we solve for the first component of the adjoint eigenvector
\beq
\label{psi1dag}
\Psi_{1,k}^{\dagger}= [I+\epsilon^{-1}L_{11}^{-e}J_r^{\dagger T} (I+\cN_0^{\dagger})^{-1}\otimes J_l^{\dagger}\cdot]L_{11}^{-e}J_{21} \phi_k',
\eeq
where $J_l^{\dagger}=J_l=\vec{\xi}$,
\be
J_r^{\dagger}:=\left(\alpha_{11}\Phi_2^{\alpha_{12}}\Phi_1^{\alpha_{11}-1}-\alpha_{12}J_{21} \tL_{22}^{-1}\tpi_{22}\Phi_2^{\alpha_{12}-1}\Phi_1^{\alpha_{11}}\right) \vec{\chi},
\eeq
and the $N\times N$ matrix $\cN_0^\dag$ is given by
\be
\cN_0^{\dag}=\epsilon^{-1}\otimes J_l^{\dagger}\cdot L_{11}^{-e}J_r^{\dag T}.
\ee

It remains to bound the second term in (\ref{psi2dag}) and the whole right-hand side of (\ref{psi1dag}). We first address the latter, where employing the bound \eqref{5lemma1}, we find
\be
\|L_{11}^{-e}J_{21} \phi_k'\|_{W^{1,1}_{\xi}} \leq C\left( \epsilon|\otimes \vec{\chi} \cdot J_{21} \phi_k'| +\epsilon^2\|J_{21} \phi_k'\|_{L^1_{1,\vec{p}}} \right),
\ee
with 
\be
J_{21} \phi_k' = \alpha_{21}q_k^{\alpha_{21}-1}\phi_k^{\alpha_{22}}\phi_k'+O(\epsilon^r).
\ee
Due to even-odd parity, this term has algebraically small mass and we deduce that 
\be
|\otimes \vec{\chi} \cdot J_{21} \phi_k'|=O(\epsilon^r).
\ee
On the other hand, $J_{21} \phi_k'$ decays exponentially away from $x=p_k$ and hence $\|J_{21} \phi_k'\|_{L^1_{1,\vec{p}}}=O(1)$.
As a consequence of these two facts we find that
\be
\|L_{11}^{-e}J_{21} \phi_k'\|_{W^{1,1}_{\xi}} \leq C\epsilon^2.
\ee
However, $J_r^\dag$ is uniformly bounded in $L^1$ since the small eigenspace of $L_{22}$ is projected away by $\tpi_{22}$. 
Using \eqref{1lemma1} we deduce that
\be
\|\epsilon^{-1}L_{11}^{-e}J_r^{\dagger T}\|_{W^{1,1}_{\xi}} \leq  C\|J_r^{\dagger T}\|_{L^1} \leq  C.
\ee
Finally, the matrix $I+\cN_0^\dag=(I+\cN_0)^\dag$ is boundedly invertible  since $\cK$ is admissible.
Taking the $W^{1,1}_{\xi}$ norm of \eqref{psi1dag}, we  conclude that
\be
\|\Psi_{1,k}^{\dagger}\|_{W^{1,1}_{\xi}} \leq C\epsilon^2.
\ee
The remainder of \eqref{adjointest} follows by applying the estimates above to \eqref{psi2dag}. \hfill $\Box$

\section{Resolvent and Semi-Group Estimates}
In this section we generate resolvent and semi-group estimates for the reduced operator $\tL_\vp$, where we have
chosen an admissible class $\cK$ of semi-strong $N$-pulses with associated contour
$\mathcal{C} \subset \mathbb{C}$, as defined in \eqref{contour}.

\begin{lemma}
For $\lambda\in\cC$, any fixed $F=(f_1,f_2)^T \in L^2({\bf R})$, and $\lambda\in\cC$, denoting the
action of the resolvent of $\tL$ on $F$ by $(\tL-\lambda)^{-1}F= (g_1,g_2)^T,$ then
\beq
\label{g1res}  g_1=(L_{11}^e+\lambda)^{-1}\left(\epsilon^{-1}J_l^T(I+\cN_{\lambda})^{-1}\otimes J_r\cdot(L_{11}^e+\lambda)^{-1}-I\right)KF.
\eeq
and
\beq \label{g2res}
  g_2 =  \left(\tL_{22}-\lambda\right)^{-1}(f_2-J_{21}g_1),
\eeq
where $J_l$, $J_r$, and $\cN_\lambda$ are defined in \eqref{J_l}, \eqref{J_r}, and \eqref{N} respectively, and 
\be
\label{K} KF=f_1-\epsilon^{-1}J_{12}(\tilde{L}_{22}-\lambda)^{-1}f_2.
\ee
\end{lemma}
\begin{proof} We recall the form \eqref{tildeL} of $\tL$ and write the resolvent problem as $(\tL-\lambda)(g_1,g_2)^T = (f_1,f_2)^T$.
For $\lambda \in \cC$ the operator $\tL_{22}-\lambda$ is invertible. Solving for $g_2$ and following the derivation of \eqref{Psi1}, we 
rewrite the first equation as,
\be
\label{g1res2} (L_{11}^e+\lambda)g_1+\epsilon^{-1}J_l^T\otimes J_r \cdot g_1=-KF.
\ee
We invert $L_{11}^e+\lambda$ and then project with the finite-rank operator $\otimes J_r$, to obtain the matrix system
\be\label{g1res3}
\otimes J_r \cdot g_1+\epsilon^{-1}\otimes J_r\cdot(L_{11}^e+\lambda)^{-1}J_l^T\otimes J_r \cdot g_1=-\otimes J_r\cdot(L_{11}^e+\lambda)^{-1}KF.
\ee
Recalling the matrix $\cN_{\lambda}$ from \eqref{N}, we may re-write this expression as
\be
\left(I+\cN_{\lambda}\right) \otimes J_r \cdot g_1=-\otimes J_r\cdot(L_{11}^e+\lambda)^{-1}KF.
\ee
Since $\lambda\notin\sigma_{ss}$ we may invert $I+\cN_\lambda$ to solve for $\otimes J_r\cdot g_1$.
Substituting this expression into \eqref{g1res3} and isolating $g_1$, we establish the closed form expression \eqref{g1res}.
\hfill \end{proof}

\subsection{Resolvent Estimates}
 Let $\cK$ be an admissible set of $N$ pulse configurations and let $\pi_\vp$ be the projection onto the small-eigenvalue eigenspace, $\sigma_0$,  of $\tL_\vp$, see \eqref{pi-def}, and let $\tpi_\vp$ denote the complementary projection. We recall the norm $\| \cdot \|_X$ introduced in 
 \eqref{Xnorm}, and the subspace $X_\vp$ corresponding to the range of $\tpi_\vp.$

\begin{proposition}
\label{prop-resolv}  There exists $C>0$ such that for all $\lambda$ in $\mathcal{C}$,  $F\in X_{\vec{p}}$, and $\vp \in \cK$, 
we have the following resolvent estimates for $\tL_\vp$,
\begin{eqnarray}
\label{propres1} \|(\tilde{L}-\lambda)^{-1}F\|_{X}& \leq& C\dfrac{\eps}{\Re(k_{\lambda})}\left(\eps\|f_1\|_{L_1}+\|f_2\|_{L^1_{\oalpha,\vp}}\right), \\
\label{propres2}  \|(\tilde{L}-\lambda)^{-1}F\|_{X} &\leq& C\dfrac{\eps}{\Re(k_{\lambda})}\left( \dfrac{\eps}{|k_{\lambda}|} \|f_1\|_{W_{\xi}^{1,1}}+
          \|f_2\|_{L^1_{\oalpha,\vp}} \right).
\end{eqnarray}
If in addition the coarse-grained projection of $f_1$ is small, then we have the enhanced residual estimate
\beq
\label{propres3}  \|(\tilde{L}-\lambda)^{-1}F\|_{X} \leq C\dfrac{\eps}{\Re(k_{\lambda})} 
\left( \left(\eps|\otimes \vec{\chi} \cdot f_1|+\epsilon^2\|f_1\|_{L^1_{1,\vec{p}}}+\|f_2\|_{L^1_{\oalpha,\vp}}\right)  \right).
\eeq
\end{proposition}
Proof:  Using the notation
\be
(\tilde{L}-\lambda)^{-1}F=\left( \ba c g_1 \\ g_2 \ea \right),
\ee
we apply the $W_{\xi}^{1,1}$ norm to $g_1$ as represented in \eqref{g1res} and use the estimate \eqref{1lemma1},
\be \label{g_1W11} \|g_1\|_{W_{\xi}^{1,1}} \leq C\dfrac{\eps}{\Re(k_{\lambda})}\|J_l^T(I+\cN_{\lambda})^{-1}\otimes J_r\cdot(L_{11}^e+\lambda)^{-1}KF\|_{L^1} + C\dfrac{\epsilon^2}{\Re(k_{\lambda})}\|KF \|_{L^1}.
\ee
Since $\cK$ is admissible, we have the estimate \eqref{Nlam-est} on  $I+\cN_\lambda$ for $\lambda\in\cC.$ 
Together with the uniform bound on $\|J_l\|_{L^1}=\vec{\xi}$, we have
\be
\label{J_lbnd1} \|J_l^T(I+\cN_{\lambda})^{-1}\|_{L^1} \leq C\left(1+\frac{\eps}{|k_\lambda|}\right)^{-1}.
\ee
Contained within $\otimes J_r$ is $(\tilde{L}_{22}-\lambda)^{-1}$ which is uniformly invertible from $L^2$ to $H^1$ for $\lambda \in \mathcal{C}$ since $F \in X_{\vec{p}}$.  Together with  \eqref{3lemma1}, these observations afford the estimates
\be
\label{J_rbnd1} |\otimes J_r\cdot(L_{11}^e+\lambda)^{-1}KF| \leq \|J_r\|_{L_1}\|(L_{11}^e+\lambda)^{-1}KF\|_{L^{\infty}} \leq C\dfrac{\epsilon^2}{|k_{\lambda}|}\|KF\|_{L^1},
\ee
Combining \eqref{J_lbnd1}, \eqref{J_rbnd1}, and \eqref{g_1W11} we find
\be
\|g_1\|_{W_{\xi}^{1,1}} \leq C\dfrac{\epsilon^2}{\Re(k_{\lambda})}\|KF\|_{L^1}.
\ee
Estimating the right hand side, we have that
\be
\|KF\|_{L^1} \leq \|f_1\|_{L^1}+\epsilon^{-1}\|J_{12}(\tilde{L}_{22}-\lambda)^{-1}f_2\|_{L^1}.
\ee
Furthermore,
\be
\|J_{12}(\tilde{L}_{22}-\lambda)^{-1}f_2\|_{L^1} \leq \|J_{12}\|_{L^2}\|(\tilde{L}_{22}-\lambda)^{-1}f_2\|_{L^2} \leq C\|f_2\|_{L^2},
\ee
which leads to the bound
\be
\label{g1bnd} \|g_1\|_{W_{\xi}^{1,1}} \leq C\dfrac{\eps}{\Re(k_{\lambda})}\left(\eps\|f_1\|_{L^1}+\|f_2\|_{L^2}\right).
\ee
Taking the $H^1_{\oalpha,\vp}$ norm of both sides of \eqref{g2res} and applying \eqref{tL22-est}, we obtain the bound
\begin{eqnarray}
 \|g_2\|_{H^1_{\oalpha,\vp}} = \|(\tilde{L}_{22}-\lambda)^{-1}(f_2-J_{21}g_1)\|_{H^1_{\oalpha,\vp}} &\leq &
C(\|f_2\|_{L^1_{\oalpha,\vp}}+\|J_{21}g_1\|_{L^1_{\oalpha,\vp}}), \nonumber \\
 & \leq &C(\|f_2\|_{L^1_{\oalpha,\vp}}+\|g_1\|_{W_{\xi}^{1,1}}).\label{g2bnd}
\end{eqnarray}
Combining \eqref{g1bnd} and \eqref{g2bnd} we obtain \eqref{propres1}.  To obtain 
\eqref{propres2}, we  take the $W_{\xi}^{1,1}$ norm both sides of \eqref{g1res} and then split the right-hand side into two parts,
\begin{align}
\|g_1\|_{W_{\xi}^{1,1}} \leq& \|(L_{11}^e+\lambda)^{-1}\epsilon^{-1}J_l^T(I+\cN_{\lambda})^{-1}\otimes J_r\cdot(L_{11}^e+\lambda)^{-1}\|_{W_{\xi}^{1,1}}+ \|(L_{11}^e+\lambda)^{-1}KF\|_{W_{\xi}^{1,1}},\nonumber\\
 \leq& C\dfrac{\eps}{\Re(k_{\lambda})}\|J_l^T(I+\cN_{\lambda})^{-1}\otimes J_r\cdot(L_{11}^e+\lambda)^{-1}KF\|_{L^1}+ \|(L_{11}^e+\lambda)^{-1}KF \|_{W_{\xi}^{1,1}},
\end{align}
where we applied \eqref{1lemma1} to obtain the second line of the inequality. We address the first factor of the first term on the right-hand side,
\beq\label{misc-bnd}
|\otimes J_r\cdot(L_{11}^e+\lambda)^{-1}KF|  \leq \|J_r\|_{L^1}\|(L_{11}^e+\lambda)^{-1}KF\|_{L^{\infty}}
\leq C\|(L_{11}^e+\lambda)^{-1}KF\|_{W_{\xi}^{1,1}}.
\eeq
Recalling both \eqref{J_lbnd1} and \eqref{g2bnd}, while using \eqref{1lemma1} to bound the $KF$ term on the right-hand side of \eqref{misc-bnd},
we find \eqref{propres2}.  For the estimate \eqref{propres3}, the bound on the $f_2$ component is as before; so without loss of
generality we consider the case $F=(f_1,0)^T$. Taking the $W_{\xi}^{1,1}$ norm of \eqref{g1res}, we have
\beq
\|g_1\|_{W_{\xi}^{1,1}} \leq
 C\left(\epsilon^{-1}\|(L_{11}^e+\lambda)^{-1}J_l^T\|_{W_{\xi}^{1,1}}\left|\otimes J_r\cdot(L_{11}^e+\lambda)^{-1}f_1\right|
+\|(L_{11}^e+\lambda)^{-1}f_1\|_{W_{\xi}^{1,1}}\right).
\eeq
Using \eqref{6lemma1} and the uniform $L^1$ bound on $J_r$, we obtain
\be
|\otimes J_r\cdot(L_{11}^e+\lambda)^{-1}f_1| \leq \|J_r\|_{L^1}\|(L_{11}^e+\lambda)^{-1}f_1\|_{L^\infty} \leq C\left( \dfrac{\epsilon^2}{|k_{\lambda}|}|\otimes \vec{\chi} \cdot f_1| +\epsilon^2\|f_1\|_{L^1_{1,\vec{p}}} \right).
\ee
From \eqref{1lemma1}, we have the bound
\be
\|(L_{11}^e+\lambda)^{-1}J_l^T\|_{W_{\xi}^{1,1}} \leq C\dfrac{\epsilon^2}{\Re(k_{\lambda})}\|J_l^T\|_{L^1} \leq C\dfrac{\epsilon^2}{\Re(k_{\lambda})}.
\ee
Finally, applying \eqref{5lemma1} to the remaining term,
\be
\|(L_{11}^e+\lambda)^{-1}f_1\|_{W_{\xi}^{1,1}} \leq C\left( \dfrac{\epsilon^2}{\Re(k_{\lambda})}|\otimes \vec{\chi} \cdot f_1| +\epsilon^2\dfrac{|k_{\lambda}|}{\Re(k_{\lambda})}\|f_1\|_{L^1_{1,\vec{p}}} \right).
\ee
Combining these estimates, we have \eqref{propres3}. \hfill $\Box$

\subsection{Semi-group estimates}
Let $\cK$ be an admissible set of $N$-pulse positions with respect to the contour $\cC$. For fixed $\vp \in \cK$, the bounds on the point and essential spectrum of $\tL_\vp$, together 
with the resolvent estimates of Proposition\,\ref{prop-resolv} show that the operator is sectorial, and from classical result, \cite{Henry}, we can generate the semi-group, $S_\vp(t)$, from the Laplace transform of the resolvent of $\tL_\vp$. Except for the $N$ point spectrum eigenvalues, 
$\sigma_0$, the spectrum of $\tL$ lies entirely on the left-hand side of $\cC$ and the semi-group $S_\vp$ associated to $\tL_\vp$ is given by the contour integral
\be
\label{semig} S_\vp(t)F=\frac{1}{2\pi i}\int_{\mathcal{C}}e^{\lambda t}(\lambda-\tL_\vp)^{-1}F d\lambda,
\ee
for all $F \in X_{\vec{p}}$.  The following estimates hold on the semi-group.

\begin{proposition} \label{P:semigroup} For any $t_0 >0$, there exists $C>0$ such that for all $\vec{p} \in \cK_{\ell,\nu},$ $F \in X_{\vec{p}}$, and $t \geq t_0$ the semi-group satisfies
\begin{align}
\label{semigroup1} \|S(t)F\|_X \leq& C e^{-\epsilon^{\alpha}\nu t}\left(\eps\|f_1\|_{L_1}+\|f_2\|_{L^1_{\oalpha,\vp}}\right), \\
\label{semigroup2} \|S(t)F\|_X \leq& Ce^{-\epsilon^{\alpha}\nu t}\left(|\ln\eps| \|f_1\|_{W^{1,1}_{\xi}}+\|f_2\|_{L^1_{\oalpha,\vp}}\right) \leq 
Ce^{-\epsilon^{\alpha}\nu t}|\ln\eps|\|F\|_{X}.
\end{align}
If in addition the coarse-grained projection of $f_1$ is small, then we have the improved estimate
\be
\label{semigroup3} \|S(t)F\|_X \leq Ce^{-\epsilon^{\alpha}\nu t}\left(\eps |\otimes\vec{\chi} \cdot f_1|+\epsilon^2\|f_1\|_{L_{1,\vec{p}}^1}+\|f_2\|_{L^1_{\oalpha,\vp}}\right).
\ee
\end{proposition}
\begin{proof}
Applying each of the resolvent estimates \eqref{propres1}-\eqref{propres3} from Proposition 5.1 to \eqref{semig} and using the integral bounds
\eqref{int1} we find \eqref{semigroup1}-\eqref{semigroup3} respectively. 
\hfill \end{proof}

The following lemma affords estimates of the integral of the resolvent over the contour $\cC.$
\begin{lemma}
Fix $p_1,p_2>0$, $t_0>0$ and let the contour $\cC$ be as in \eqref{contour}, then there exists  $C>0$ such that for all $t>t_0$
\beq
\label{int1} \int_{\mathcal{C}}\Frac{ |e^{\lambda t}|}{|\Re(k_{\lambda})|^{p_1}|k_\lambda|^{p_2}}d\lambda \leq Ce^{-\epsilon^{\alpha}\nu t} \eps^{-(p_1+p_2)}\left\{ \begin{array}{lcl} 1 &{\rm if}& p_1+p_2<2, \\
                                                                |\ln \eps| &{\rm if}& p_1+p_2=2,\\
                                                                \eps^{-\frac{\alpha}{2}(p_1+p_2-2)} &{\rm if} & p_1+p_2>2.\end{array}\right.
\eeq
\end{lemma}
\begin{proof}
Estimates on the angled parts $\mathcal C_{l^\pm}$ are straightforward because of the exponential decay in $e^{\lambda t}$.  The concern is for 
the vertical part, $\cC_v$ where for $\lambda=-\eps^\alpha\nu +is$, we have
\beq
C_- \eps \left(\epsilon^{2\alpha}\nu^2+s^2\right)^{\frac14}\leq|k_{\lambda}| +|\Re(k_{\lambda})| \leq C_+ \eps \left(\epsilon^{2\alpha}\nu^2+s^2\right)^{\frac14},
\eeq
for some constants $C_\pm>0$.
Applying these estimates \eqref{int1} we have
\beq
 \int_{\mathcal{C}}\Frac{ |e^{\lambda t}|}{|\Re(k_{\lambda})|^{p_1}|k_\lambda|^{p_2}}d\lambda \leq C e^{-\epsilon^{\alpha}\nu t}
 \eps^{-(p_1+p_2)}\int_{0}^{b}\Frac{ds}{(s^2+\eps^{2\alpha}\nu^2)^\frac{p_1+p_2}{4}}.
 \eeq 
For $p_1+p_2<2$ the integral is uniformly bounded, even if $\mu=0$. For $p_1+p_2\geq2$ we rescale $s$ by $\eps^\alpha$ and bound
the resulting integral.
\hfill \end{proof}
\begin{remark}
The semi-group estimates are derived for $t>t_0$. Short time estimates can be derived which scale like $t^{-1/2}$ as $t\to0.$
This integrable singularity has no impact on the analysis which follows, and we omit it.
\end{remark}

\section{Nonlinear Adiabatic Stability by Renormalization Group}
Our primary result is that, in an $\|\cdot\|_X$ neighborhood of each admissible $N$-pulse manifold $\cK$, we may
decompose solutions $\vU=(U,V)^T$ of \eqref{UV} as
\beq
\label{decompU,V} \vU = \Phi_{\vec{p}} + W^*(x,t),
\eeq
where the $N$-pulse configuration $\vp=\vp(t)$ is dynamic in time and the remainder $W^*=(W_1^*,W_2^*)^T$ can be uniformly controlled in the $\|\cdot\|_X$ norm. In particular, we prove the adiabatic stability result, \eqref{Main-bound}, and derive the limiting pulse dynamics, \eqref{Main-pulse}. 
The pulse profiles, $\Phi$, are only approximate equilibria.  Inserting the decomposition \eqref{decompU,V} into \eqref{UV} and Taylor expanding $\cF$ about $\Phi_\vp$ yields
\be
\label{decompU} W^*_t +\nabla_\vp\,  \Phi \cdot\dot{\vp}=\cF(\Phi)+L_{\vp}W^*+\cNon(W^*),
\ee
where the linear operator $L_\vp$ was introduced in \eqref{lin-op} and the nonlinearity $\cNon= (\cNon_1,\cNon_2)^T$  is 
\beq\label{defN}
\cNon(W^*) := \cF(\Phi+W^*)-\cF(\Phi)-L_\vp\, W^*.
\eeq

The renormalization group procedure starts by freezing the reference pulse configuration $\vec{p}=\vec{p}_0$ in $X_{\vec{p}_0}$, where $\vec{p}_0$ is as constructed in Proposition 6.1. The semi-group estimates of  Proposition\,\ref{P:semigroup} require a linear operator with a frozen pulse configuration,
so that we replace $L_\vp$ not merely with its reduced linearization, but with its reduced linearization at a frozen pulse configuration, $\tL_{\vp_0}.$ 
Moreover, we separate the remainder $W^*$ into two parts, $W^*=W+{\Phi}_c(\cdot,\vp,\vp_0)$, where the {\sl correction term}, $\Phi_c$, 
which serves to adjust the shape of the pulse profile, depends upon both the fast pulse configuration, $\vp=\vp(t)$, and the frozen reference $\vp_0.$    More specifically,
we have the decomposition
\beq\label{decompPhic}
\vU = \Phi_{\vec{p}}+ {\Phi}_c(\vp,\vp_0)+W, 
\eeq
where the correction term,  ${\Phi}_c= \left(\Phi_{c,1},\Phi_{c,2}\right)^T$ is chosen to cancel the parts of the residual, $\cF(\Phi)$, which do 
not contribute to the pulse dynamics,
\be
\label{vecPhi1} \Phi_c(\cdot; \vp\,, \vp_0) := -\tL_{\vp_0}^{-1}\tpi_{\vp_0}\cF\big(\Phi(\cdot;\vp\,)\bigr).
\ee
Here $\tL_{\vp_0}$ is the reduced operator defined in \eqref{tildeL}, and $\tpi_{\vp_0}$ is the orthogonal spectral projection off of its
small eigenvalue eigenspace, $\sigma_0(\vp_0)$, defined in \eqref{tildepi}.  The complementary part of the remainder, $W=(W_1,W_2)^T$, 
incorporates errors that arise from the transient nature of the flow - parts which may not necessarily be slaved to the pulse configurations.
  
The following proposition constructs a base point $\vp_0$ about which the local coordinate system is developed.
\begin{proposition} 
\label{P:basepoint} Let $\cK$ be an admissible class of $N$-pulses. Then there exist $\delta$ sufficiently small, 
$M_1 > 0$, and a smooth function $\cJ:X \to \cK$ such that for  all  $\vU_0=(U_0,V_0)^T$ satisfying
\be
\|\vec{U}_0-\left(\Phi(\vp_*)+\Phi_c(\vp_*,\vp_*)\right)\|_X \leq \delta,
\ee
where $W_*:=\vU_0-(\Phi(\vp_*)+\Phi_c(\vp_*,\vp_*))$, for some $\vec{p}_* \in \cK_\delta$, then $\vp=\vp_*+\cJ(W_*)$ satisfies
\be
\label{prop6state} W_0 := \vec{U}_0 - (\Phi(\vp)+\Phi_c(\vp,\vp)) \in X_{\vec{p}},
\ee
for $X_{\vec{p}}$ defined in \eqref{Xvp}.  Moreover, if $W_*:=\vU_0-(\Phi(\vp_*)+\Phi_c(\vp_*,\tp)) \in X_{\tp}$ for some $\tp \in \cK_\delta$, then
\be
\label{mvtbnd} |\vec{p}-\vec{p}_*| \leq M_1\|W_*\|_X|\vec{p}_*-\tilde{p}|.
\ee
\end{proposition}
\begin{proof}  We may write $\vU_0=\Phi(\vp_*)+\Phi_c(\vp_*,\vp_*)+W_*$ and $\vU_0=\Phi(\vp)+\Phi_c(\vp,\vp)+W_0,$ and hence
\be
W_0 = W_* + \Phi(\vp_*)+\Phi_c(\vp_*,\vp_*)-\left(\Phi(\vp)+\Phi_c(\vp,\vp)\right).
\ee
Since by \eqref{vecPhi1}, $\pi_\vp \Phi_c(\vp,\vp)=0$, the equation \eqref{prop6state} is equivalent to 
solving the nonlinear system 
\beq \label{Lam-sys}
\Lambda_j(\vp,W_*;\vp_*):=\left(W_*+ \Phi(\vp_*)+\Phi_c(\vp_*,\vp_*)-\Phi(\vp),\Psi_j^{\dag}(\vp)\right)_{L^2}=0.
\ee
Writing $\Lambda=(\Lambda_1,\cdots, \Lambda_N)^T$, it is clear that $\Lambda(\vec{p}_*,0)=\vec{0}.$ In light of \eqref{Phi1Linfty} and \eqref{vecPhi_1} 
the function $\Lambda$ is smooth in $\vp, \vp_*$ and $W_*$. We examine the following gradient
\be
\label{gradientmat} \left[\nabla_{\vp}\Lambda|_{(\vec{p}=\vec{p}_*,W_*=0)}\right]_{ij}=-\left( \Frac{\partial \Phi(\vp)}{\partial p_i},\Psi_j^{\dag}(\vp)\right)_{L^2}.
\ee
We evaluate the right-hand side using \eqref{Phi1Linfty}, \eqref{dxPhi1Linfty}, and \eqref{adjointest}, and obtain
\be
\label{gradientmat2} \nabla_{\vec{p}}\Lambda|_{\vec{p}=\vec{p}_*,W_*=0}=-\|\phi'\|^2_{L^2}Q^{-2\alpha_{21}/(\alpha_{22}-1)}+O(\eps),
\ee
where $Q$ is the diagonal matrix of $\vq$ amplitudes, whose determinate is bounded away from zero for $\vp\in\cK$.
The existence of $\cJ$ verifying \eqref{prop6state} follows from the implicit function theorem.
If, in addition, we have $W_*:=\vU_0-(\Phi(\vp_*)+\Phi_c(\vp_*,\tp)) \in X_{\tp}$, then we may re-write \eqref{Lam-sys} as
\beq
\left( \Phi(\vp_*)+\Phi_c(\vp_*,\tp)-\Phi(\vp),\Psi_j^{\dag}(\vp)\right)_{L^2}=-\left(W_*,\Psi_j^\dag(\vp)\right)_{L^2}=
-\left(W_*,\Psi_j^\dag(\vp)-\Psi_j^\dag(\tp)\right)_{L^2}=:\Xi_j,
\eeq
where the last equality follows from the assumption $W_*\in X_{\tp}.$ We use \eqref{gradientmat2} to Taylor expand the left-hand side, obtaining
the leading order relation
\beq 
\vp-\vp_* = -\Frac{1}{\|\phi^\prime\|^2_{L^2}} Q^{2\alpha_{21}/(\alpha_{22}-1)} \Xi + O(|\vp-\vp_*|^2).
\eeq
However, we readily see from Lemma \ref{L:Psi-dag} that $|\Xi| \leq C \|W_*\|_X|\vp-\tp|$, which establishes \eqref{mvtbnd}.
\hfill \end{proof}

\subsection{Projected equations}
We emphasize that the pulse configuration $\vp=\vp(t)$ in $\Phi=\Phi(\cdot,\vp(t))$ is {\em not} frozen, but evolves freely with $t$.  The 
separation of the pulse configuration into fast and slow (frozen) variables requires the introduction of the secular operator, 
$\Delta L := L_{\vp} - \tL_{\vp_0}$, which induces a  temporally growing forcing term as $\vp=\vp(t)$ evolves away from $\vp_0$. 
Eventually the secular term, $\vp(t)-\vp_0$, grows too large, and the equations are re-projected, using Proposition\,\ref{P:basepoint} to update the frozen
base-point $\vp_0$ to $\vp_1$ and $W_0\in X_{\vp_0}$ to $W_1\in X_{\vp_1}$.  In the sequel, we show that the long-time evolution of both pulse 
configuration and  remainder is well described by the collection of initial data $\{\vp_n\}_{n=0}^\infty$ and $\{W_n\}_{n=0}^\infty.$ 

As the first step in this process, we rewrite \eqref{decompU} as
\begin{align}
\label{decomp2U}  W_t +\left(\nabla_\vp\,  \Phi +\nabla_\vp\,  \Phi_c\right)  \cdot\dot{\vp}&=\pi_{\vp_0}\cF(\Phi)+\tL_{\vp_0}W +\Delta L \left({\Phi}_c+W\right)+\cNon(\Phi_c+W),\\
W(x,0)&=W_0,
\end{align}
where $\vp_0$ and $W_0\in X_{\vp_0}$ are as determined in Proposition\,\ref{P:basepoint}. Moreover, following classical modulation theory, we choose
 the evolution of $\vp=\vp(t)$ to enforce the condition $W \in X_{\vec{p}_0}$. To extract the evolution of the pulse configuration, we
project \eqref{decomp2U} onto the small eigenvalue adjoint eigenspace, $\{\Psi_j^\dag(\cdot;\vp_0)\}_{j=1}^N$, whose $L^2$ orthogonal complement coincides with 
the kernel of $\pi_{\vp_0}$. Since $\pi_{\vp_0} \tL_{\vp_0}W=\tL_{\vp_0}\pi_{\vp_0}W=0$,
we attain the fast pulse evolution
\beq
\label{vp-evol}
\begin{aligned}
  M \dot \vp &= g(\vp\,;\vp_0,W),\\
  \vp\,(t_0) &=\vp_0,
\end{aligned}
\eeq
where the $N\times N$ matrix $M$ takes the form
\beq
 M_{ij} := \left(\frac{\partial \Phi}{\partial p_i}+\frac{\partial\Phi_c}{\partial p_i}, \Psi_j^\dag(\vp_0)\right)_{L^2}, \label{M-def}
\eeq
and the forcing term $g\in \bR^N$ is given by
\vskip -0.25in
\beq
  g_j:=\Bigl(\cF\bigl(\Phi(\vp\,)\bigr)+\Delta L \left(\Phi_c+W\right)+\cNon(\Phi_c+W),\Psi_j^{\dag}(\vp_0)\Bigr)_{L^2}=\Bigl({\cF(\Phi+W^*)},\Psi_j^{\dag}(\vp_0)\Bigr)_{L^2}.\label{g-def}
\eeq
The evolution of the remainder $W$ is obtained by applying the complementary spectral projection, $\tpi_{\vp_0}$,
to  \eqref{decomp2U}, which yields

\beq
\label{decomp3U}
\begin{aligned}
W_t &= R + \tL_{\vp_0}W+\tilde{\pi}_{\vec{p}_0}\left(\Delta L \left(\Phi_c+W\right)+\cNon(\Phi_c+W)\right),\\
W(x,0) &= W_0,
\end{aligned}
\eeq
where we have introduced the temporal component of the residual,
\beq\label{temp-resid}
R(\dot\vp\,; \vp_0 ):= -\tpi_{\vp_0}\left(\nabla_\vp\,\Phi+\nabla_\vp\,\Phi_c\right)\cdot \dot{\vp}.
\eeq

The central goal of the renormalization approach is to control the growth of the remainder $W$ and the secularity, $|\vp(t)-\vp_0|$, in systems \eqref{vp-evol} and \eqref{decomp3U} in terms of the quantities
\begin{align}
\label{T1} T_W(t)&:= \underset{t_0 <s < t}{\sup}e^{\epsilon^{\alpha}\nu (s-t_0)}\|W(s)\|_X, \\
\label{T2} T_\vp(t)&:= \underset{t_0 <s < t}{\sup}|\vec{p}(s)-\vec{p}_0|.
\end{align}

\subsection{Bounds on Pulse Dynamics}
The following lemma verifies that the matrix $M$ is uniformly invertible.
\begin{lemma}
The matrix $M$ defined in \eqref{M-def} has the following asymptotic form
\beq \label{M-asymp}
  M =-\|\phi_0'\|^2_{L^2}Q^{\frac{-2\alpha_{21}}{\alpha_{22}-1}}+O(\epsilon^{1-\alpha/2}),
 \eeq
 where $Q$ is the diagonal matrix of amplitudes $\vq$.
 \end{lemma}
\begin{proof}
 From \eqref{adjointest} we see that $\Psi_j^\dag$ is dominated by its second component, which satisfies 
 $\Psi_{2j}^\dag=\phi_j^\prime+O(\eps^{1+\alpha/2})$. Similarly, from \eqref{dxPhi1Linfty} we see that 
 $\frac{\partial\Phi_2}{\partial p_j} = -\phi_j^\prime +O(\eps^{1+\alpha/2}).$ From \eqref{vecPhi_1} we have the bound
 $\|\nabla_\vp\, \Phi_c\|_X\leq C\eps^{1-\alpha/2},$ which shows the contribution from this term is higher order. Recalling the scaling \eqref{phij-phi0}
 yields \eqref{M-asymp}.
\hfill\end{proof}

The lemma below establishes an upper bound on the rate of pulse motion.
\begin{lemma}
The pulse evolution, given by \eqref{vp-evol} satisfies the following bound
\beq\label{vp-dot-bd}
|\dot{\vp}\,|\leq C \left( \eps+ (\eps+T_\vp(t)) (\|W(s)\|_X+\eps^{1-\alpha/2})+ \|W(s)\|_X^2+\eps^{2-\alpha}\right).
\eeq
\end{lemma}
\begin{proof}
From \eqref{M-asymp} the matrix $M$ is boundedly invertible. We apply $L^1-L^\infty$ estimates to the first component of $g$ and $L^2-L^2$ estimates to the 
second component,
\beq
 |\dot{\vp}_j\,| \leq C \left( \|\cF_1(\Phi+W^*)\|_{L^1} \|\Psi_{j,1}^\dag\|_{L^\infty}+ \|\cF_2(\Phi+W^*)\|_{L^2} \|\Psi_{j,2}^\dag\|_{L^2}\right),
 \eeq
for $\cF(\Phi+W^*)=\left(\cF_1(\Phi+W^*),\cF_2(\Phi+W^*)\right)^T$ defined in \eqref{g-def}. From \eqref{adjointest} we have  $\|\Psi^\dag_{j,1}\|_{L^\infty}\leq C\eps^2$ while $\|\Psi_{j,2}^\dag\|_{L^2} = O(1)$.
  Combining these estimates with \eqref{3lemma2}, \eqref{DeltaLsW}, \eqref{DeltaLrW}, \eqref{nonlinestimate}, and the bound $\|W^*\|_X\leq C(\|W\|_X+\eps^{1-\alpha/2})$
afforded by \eqref{vecPhi_1} yields \eqref{vp-dot-bd}. \hfill
\end{proof}

\subsection{Decay of the Remainder}
The following proposition establishes uniform estimates on the decay of $\|W\|_X$ over the duration, $\Delta t$,
that the linearized operator $\tL=\tL_{\vp_0}$ is fixed at $\vp_0$. 
It is convenient to introduce the decay factor
\beq\label{beta-def}
\beta(\Delta t) := e^{-\eps^{\alpha} \nu \Delta t},
\eeq
which bounds the action of the semi-group $S_{\vp_0}(\Delta t)$ on $X_{\vp_0}$.
\begin{proposition}
\label{P:W-decay}
Fix $\eps_0>0$ sufficiently small and let the normal hyperbolicity condition, $0<\alpha<\frac12,$ hold. There exist constants $C_1,C_2>0$ such that
for all $0<\eps\leq \eps_0$ and all initial data $W_0\in X_{\vp_0}$ satisfying
\beq  \|W(t_0)\|_X \leq C_1\frac{ \eps^{\alpha}}{|\ln\eps|},\label{W0-cond}
\eeq
then the solution $W$ to \eqref{decomp3U} satisfies
\beq
 \|W(t)\|_X \leq C_2 \left( \|W(t_0)\|_X |\ln \eps|e^{-\eps^{\alpha}\nu(t-t_0)} +\eps^{1-\alpha}\right), 
 \label{W-decay1}
 \eeq
 for all $t\in[t_0,t_0+\Delta t]$ for any $\Delta t$ for which the decay factor meets
 \beq
  \beta(\Delta t) \geq C_1 \eps^{1-\frac32\alpha} .\label{Dt-cond}
 \eeq
 In particular, we may choose $\Delta t$ so that at $t_1:=t_0+\Delta t$ we have
 \begin{eqnarray}
   \label{W-decay2}
  \|W(t_1)\|_X &\leq& C_2 \left( \|W(t_0)\|_X|\ln\eps|\eps^{1-3\alpha/2} +\eps^{1-\alpha}\right),\\
  |\vp(t_1)-\vp_0|& \leq & C_2 |\ln\eps|\eps^{1-\alpha}.\label{vp-gap}
 \end{eqnarray}
\end{proposition}
\begin{proof}
Applying the variations of constants formula to \eqref{decomp3U}, we have
\be
\label{varcons} W(x,t) = S_{\vp_0}(t-t_0)W_0 + \int_{t_0}^t S_{\vp_0}(t-s)\Bigl(R+\tpi_{\vp_0}\bigl(\Delta L  (\Phi_c+W)+\cNon(\Phi_c+W)\bigr)\Bigr)ds,
\ee
where $S_{\vp_0}$ is the semi-group generated by $\tL_{\vp_0}$.  Applying \eqref{Phi1Linfty}, \eqref{dxPhi1Linfty}, and \eqref{vecPhi_1} to \eqref{temp-resid}, we find
\beq \eps^{\alpha/2}\|R_1\|_{L^1}+\|R_2\|_{L^{1}_{\oalpha,\vp\,}} \leq C|\dot{\vp}\,|.
\label{lemma4}
\eeq
  We take the $X$-norm of \eqref{varcons},  and apply the semi-group 
estimate \eqref{semigroup1} to $R$ and \eqref{semigroup2} to $W_0$ while using \eqref{semi11}-\eqref{semi13} and
\eqref{vecPhi_1}, on the remaining terms. Powers of $\|W\|_X$ appear from several terms, but
 the dominant contribution is from the temporal residual, $|\dot{\vp}|$, via \eqref{vp-dot-bd}, yielding
\beq
\|W(t)\|_X \leq C\left( e^{-\epsilon^{\alpha}\nu(t-t_0)}|\ln \eps|\|W(t_0)\|_X+ 
\int_{t_0}^t  e^{-\epsilon^{\alpha}\nu (t-s)} |\dot{\vp}\,|\, ds\right). \label{toto1}
\eeq
We evaluate \eqref{toto1} at $t=t'$, multiply by $e^{\epsilon^{\alpha}\nu (t'-t_0)}$,  use \eqref{vp-dot-bd} to
control $\dot{\vp}$ and bound $\|W(s)\|_X\leq e^{-\eps^{\alpha}\nu(s-t_0)}T_W(t)$ valid for $t>s$. Taking the $\sup$
over $t' \in (t_0,t)$, we obtain
\begin{eqnarray} 
T_W(t) &\leq &C|\ln \eps|T_W(t_0) +C\int_{t_0}^t \Bigl [e^{\eps^{\alpha}\nu(s-t_0)}(\eps+\eps^{2-\alpha})+ \nonumber\\
&& \hspace{0.5in} (\eps+T_\vp(t))\left(T_W(t)+\eps^{1-\alpha/2} e^{\eps^{\alpha}\nu(s-t_0)}\right) +
T_W^2(t)e^{-\eps^{\alpha}\nu(s-t_0)}\Bigr]ds.\label{T1ineq0} 
\end{eqnarray}
Evaluating the integrals on the right-hand side, recalling the decay factor $\beta$ from \eqref{beta-def}, and keeping dominant terms, yields
\beq
T_W(t) \leq C\left(|\ln\eps|T_W(t_0) +\frac{\eps^{1-\alpha}}{\beta}+\left(\eps+T_\vp(t)\right) \left(T_W(t) \Delta t+\frac{\eps^{1-3\alpha/2}}{ \beta}\right)+\eps^{-\alpha}T_W^2(t) \right).
\label{T1ineq} 
\eeq
where $\Delta t := t - t_0$ is the length of the renormalization interval. The following lemma bounds $T_\vp$.

\begin{lemma} \label{L:T2} Fix $\cK$, then for $\alpha<2$ there exists a constant $C_1>0$ sufficiently small, but independent of $\eps$,
and $C_2>0$ such that for all $t>t_0$ for which 
\beq T_W(t)\leq C_1\eps^\alpha \hspace{0.5in}{\rm and}\hspace{0.5in} 
          \Delta t \leq C_1 \eps^{\alpha/2-1},
     \label{C-1}
 \eeq
 then
\beq
\label{T2est} T_\vp(t) \leq C_2 \left(\Delta t (\eps+\eps^{2-\alpha})+ \eps^{1-\alpha} T_W(t)  +\eps^{-\alpha} T_W^2(t)\right).
\eeq
\end{lemma}
\begin{proof}
From the definition \eqref{T2} of $T_\vp$ we have the bound
\be
T_\vp(t) \leq \int_{t_0}^{t_0 + \Delta t} |\dot{\vp}\,(s)|ds.
\ee
Using the estimate \eqref{vp-dot-bd}, the bound $\|W(s)\|_X\leq CT_W(t) e^{-\eps^{\alpha}\nu(s-t_0)}$ valid for $t_0<s<t$, and integrating in 
$s$ yields
\beq\label{RHSconst}
T_\vp(t) \leq C \left(\Delta t (\eps+\eps^{2-\alpha})+\eps^{1-\alpha}T_W(t) +T_\vp(t)(\eps^{-\alpha}T_W(t)+\Delta t \eps^{1-\alpha/2}) +\eps^{-\alpha} T_W^2(t)\right).
\eeq
Subject to the constraints \eqref{C-1}, for $C_1$ sufficiently small, the $T_\vp$ term on the right-hand side of 
\eqref{RHSconst} may be absorbed into the left-hand term. Adjusting the constant we obtain \eqref{T2est}. 
\hfill\end{proof}
\vskip 0.1in
We continue the estimation of $T_W$ required to establish Proposition \ref{P:W-decay}. Inserting the bound \eqref{T2est} of Lemma\,\ref{L:T2} into \eqref{T1ineq} and collecting powers of $T_W$ we find
\begin{align}
T_W(t) \leq C \Biggl( & \left(|\ln\eps|T_W(t_0) + \frac{\eps^{1-\alpha}+(\Delta t)\left(\eps^{2-3\alpha/2}+\eps^{3-5\alpha/2}\right)}{\beta}\right)+\nonumber\\ 
& T_W(t) \left(\frac{\eps^{2-5\alpha/2}}{\beta}+\eps\Delta t+(\Delta t)^2(\eps+\eps^{2-\alpha}) \right)+\nonumber\\
& T_W^2(t)\left(\frac{\eps^{1-5\alpha/2}}{\beta}+\eps^{-\alpha}+\eps^{1-\alpha}\Delta t\right)+\Delta t\eps^{-\alpha}T_W^3(t)\Biggr).\label{T1-est2}
\end{align}
For this estimate to be meaningful we require that the coefficient of $T_W$ on the right-hand side be strictly less than one. This generates two conditions,  the first requires
\beq
\label{beta-max}
\beta > \frac{\eps^{2-5\alpha/2}}{C}, \hspace{0.2in}{\rm equivalently} \hspace{0.2in}
\Delta t <\eps^{-\alpha}\frac{1}{\nu}\left((2-\frac52\alpha)|\ln\eps|-\ln C\right),
\eeq
and since the renormalization interval $\Delta t$ must be at least $\eps^{-\alpha}$ in order to insure that $\beta\ll 1$, we impose the bounds
\beq
     \eps^{-\alpha} < \Delta t <\eps^{-\alpha}\frac{1}{\nu}\left((2-\frac52\alpha)|\ln\eps|-\ln C\right).
\eeq
In light of these bounds, the second condition for the absorption of the $T_w$ term becomes
\beq
(\Delta t)^2\left(\eps +\eps^{2-\alpha}\right) \ll 1 \hspace{0.2in}{\rm equivalently} \hspace{0.2in}  \boxed{\alpha<\frac 12},\label{alpha-cond}
\eeq
which sets the ultimate limit on $\alpha.$ Subject to these conditions we absorb the linear term and merge the cubic term into the larger quadratic term. Enforcing
the bounds on $\Delta t$ we obtain the key estimate
\beq
T_W(t) \leq C \left( |\ln \eps| T_W(t_0) + \frac{\eps^{1-\alpha}}{\beta}+ T_W^2(t)\left(\eps^{-\alpha}+\frac{\eps^{1-5\alpha/2}}{\beta}\right)\right).
\eeq

This inequality has the form of a quadratic equation, $h(r)=0$ where $r:=T_W(t)$ and
\vskip -0.2in
\beq
h(r):= \overbrace{C\left(|\ln \eps| \,T_W(t_0) + \frac{\eps^{1-\alpha}}{\beta}\right)}^{a_0}
-r +\overbrace{C\left(\eps^{-\alpha}+\frac{\eps^{1-5\alpha/2}}{\beta}\right)}^{a_2} r^2.
\eeq
This equation has two positive roots $0 <r_1=O(a_0)\ll r_2$  if $a_0a_2\ll1$. Indeed,  taking $C_1$ in \eqref{W0-cond} sufficiently small, independent
of $\eps$,  these are precisely the conditions  imposed by \eqref{W0-cond} and \eqref{Dt-cond}, under which we have the bound
\beq\label{T1-bound}
 T_W(t) \leq r_1 \leq C \left(|\ln\eps| \,T_W(t_0) + \frac{\eps^{1-\alpha}}{\beta}\right).
 \eeq

Recalling the definition \eqref{T1} of $T_W$ we obtain \eqref{W-decay1}. Choosing the minimal value of $\beta$, equivalently the maximal value of $\Delta t$,
permitted in \eqref{Dt-cond} we find the first bound of \eqref{W-decay2}, while the second follows
from \eqref{W0-cond}. The estimate \eqref{vp-gap} is then a consequence of \eqref{T2est} and the bound \eqref{T1-bound} on $T_W$.
 \hfill\end{proof}

\subsection{The Renormalization Group iteration}
At the conclusion of the first renormalization interval $(t_0, t_1)$, where $t_1:=t_0+\Delta t$, we have a pulse configuration $\vp(t_1)$
and a remainder $W(\cdot, t_1)\in X_{\vp_0}$ whose $X$ norm is smaller than that of $W(\cdot, t_0)$, so long as $\alpha<\frac12.$ 
We are in a position to iterate the renormalization process until $\vp$ hits the boundary of $\cK.$
That is given $\vp_{n-1}\in\cK_0$, $t_n:=t_{n-1}+\Delta t_{n-1}$, $W(\cdot,t_n)\in X_{\vp_{n-1}}$ satisfying \eqref{W0-cond}, and $\vp(t_{n-1})$ satisfying \eqref{vp-gap}  then
defining 
$$U(\cdot, t_n)=\Phi(\cdot, \vp(t_n))+\Phi_c(\cdot;\vp(t_n),\vp_{n-1}) + W(\cdot, t_n),$$
 from  Proposition\,\ref{P:basepoint} we may construct $\vp_{n}=\vp(t_n)+\cJ(W(t_n))$ and $W_n(\cdot)\in X_{\vp_n}$ such that
 $$ U(\cdot, t_n) = \Phi(\cdot,\vp_n)+\Phi_c(\cdot,\vp_n,\vp_n)+ W_n(\cdot),$$
where $\Phi_c$ is defined according to  \eqref{vecPhi1}. Moreover, since $W(t_n)\in X_{\vp_{n-1}}$ we may apply
\eqref{mvtbnd} to estimate the jump function $\cJ$,
\beq
\label{jump-bd}
| \vp_n-\vp(t_n)| \leq M_1 \|W(t_n)\|_X |\vp_n-\vp_{n-1}|,
\eeq
which,  from \eqref{W-decay2}, is negligible compared to evolution of $\vp(t)$ over $t\in(t_{n-1},t_n).$ The jump in the remainder upon
the renormalization satisfies
\beq
\label{preWjump} \|W_n-W(t_n)\|_X =\left\|\bigl(\Phi(\vp_n)-\Phi(\vp(t_n))\bigr)+\bigl(\Phi_c(\vp_n,\vp_n) -\Phi_c(\vp_{n-1},\vp(t_n))\bigr)\right\|_X.
\eeq
The dominant contribution in the first term comes from the $\Phi_1$ component, which from \eqref{Phi1Linfty}
has a sensitivity to variation in $\vp$ on the order of $\eps^{-\alpha}$. Using  \eqref{vecPhi_1} on the second term yields the bound
\beq
\label{pre2Wjump}
 \|W_n-W(t_n)\|_X\leq C\left(\eps^{-\alpha}|\vp_n-\vp(t_n)|+ \eps^{1-\alpha/2}|\vp_n-\vp_{n-1}|\right).
 \eeq
The estimates \eqref{jump-bd} and \eqref{vp-gap} yield
\beq
\label{Wjump}
 \|W_n-W(t_n)\|_X\leq C\left(|\ln\eps| \eps^{1-2\alpha}\|W_{n}\|_X+|\ln\eps|\eps^{2-3\alpha/2}\right).
 \eeq
In particular, applying the triangle inequality to bound $\|W_n\|_X$ we see for $0<\alpha<\frac12$ that the dominant contribution is from $\|W(t_n)\|_X$, 
which from \eqref{W-decay2} takes the form
\beq\label{Wn-est}
\|W_n\|_X \leq C\left( |\ln\eps| \eps^{1-3\alpha/2}\|W_{n-1}\|_X+ \eps^{1-\alpha}\right).
\eeq

We introduce the Renormalization Group map 
\beq
\mathfrak{G}\left( \ba c W_{n-1} \\ \vp_{n-1}\ea\right) = \left( \ba c W_n \\ \vp_n\ea\right),
\eeq
which produces the initial data at time $t_n$ from the initial data for the pulse and remainder equations, \eqref{vp-evol} and \eqref{decomp3U}, on
the interval $(t_{n-1},t_n)$. From a simple, linear iteration argument applied to \eqref{Wn-est}, we have the estimate
\beq
   \|W_n\|_X \leq C \left( \left(|\ln\eps|\eps^{1-3\alpha/2}\right)^n \|W_0\|_X+ \eps^{1-\alpha}\right). \label{RG-its}
 \eeq
 The estimate \eqref{Main-bound} follows.

\subsection{Long-time asymptotics} From \eqref{vp-gap} we see that $T_\vp\leq C|\ln\eps|\eps^{1-\alpha}$ at each point of the RG iteration process,
while $\|W\|_X\leq T_W$. Using these bounds in \eqref{vp-dot-bd}, we see the dominant long-time contribution to the 
pulse evolution arises from the residual, with the next largest contribution from the square of  $\|\Phi_c\|_X$. From \eqref{M-asymp} we may also
invert the matrix $M$, yielding the reduced evolution
\beq 
\label{red-dyn} \dot{p}_j= -\dfrac{q_j^{2\alpha_{21}/(\alpha_{22}-1)}}{\|\phi_0^\prime\|^2_{L^2}} \left(\cF_2(\Phi), \phi_j^\prime \right)_{L^2}+O\!\left(\epsilon^{2-\alpha}, \eps\|W\|_X,\|W\|_X^2 \right),
\eeq
for $j=1,\cdots, N.$
Substituting for $\cF_2(\Phi)$ from \eqref{resasymp} and integrating by parts, we find 
\beq
\left(\cF_2(\Phi),\phi_j^\prime\right)_{L^2}= -\frac{\alpha_{21}}{\alpha_{22}+1}\left(\Phi_1^{\alpha_{21}-1}\Phi_1^\prime, \phi_j^{\alpha_{22}+1}\right)+ 
O(\eps^{2+\alpha}).
\eeq
Using the last estimate of \eqref{dxPhi1Linfty} to replace $\Phi_1$ with $q_j$, yields
\beq
\left(\cF_2(\Phi),\phi_j^\prime\right)_{L^2}= -\frac{\alpha_{21}q_j^{\alpha_{21}-1}}{\alpha_{22}+1}\left(\Phi_1^\prime, \phi_j^{\alpha_{22}+1}\right)+ 
O(\eps^{2+\alpha}).
\eeq
Turning to \eqref{Phi1-def}, we see that the odd derivatives of $\Phi_1$ at $x=p_j$ of order $3$ or higher are at most $O(\eps^{3+\alpha}),$
while the even derivatives contribute an odd component to $\Phi_1^\prime$. That is, in a neighborhood of $x=p_j$ we have
\beq
  \Phi_1^\prime(x)= \Phi_1^\prime(p_j) + \Xi_{\rm odd}(x) +O\!\left(\eps^{3+\alpha}\right),
 \eeq
where $ \Xi_{\rm odd}(x)$ is odd about $x=p_j$ and does not contribute to the inner product.
Combining these observations with the scaling \eqref{phij-phi0} and inserting into \eqref{red-dyn} shows that the pulse velocity
is proportional to the derivative of $\Phi_1$ at the pulse position,
\beq\label{dyneqexp}
 \dot{p}_j = \frac{\alpha_{21}}{\alpha_{22}+1}\frac{\|\phi_0\|_{L^{\alpha_{22}+1}}^{\alpha_{22}+1}}{\|\phi_0^\prime\|^2_{L^2}}\frac{1}{q_j}\Phi_1^\prime(p_j)+O\!\left(\epsilon^{2-\alpha}, \eps\|W\|_X,\|W\|_X^2 \right).
\eeq
To evaluate $\Phi_1^\prime(p_j)$  we invert $L_{11}^e$ in  \eqref{Phi1-defb} in terms of its Green's function $G_0$, \eqref{Ggreensfunc},
 and take the $x$-derivative of the result, and rescale $\phi_j$ according to \eqref{phij-phi0}, yielding
\begin{align}
\Phi_1^\prime(p_j) &=  -\eps^{-1}\sum_{k=1}^N q_k^\theta\int\limits_\bR \phi_0^{\alpha_{12}}(x-p_k) G_0^\prime(p_j-x))\,dx,\nonumber\\
                                   &= \eps \sum_{k=1}^N q_k^\theta \int\limits_\bR \phi_0^{\alpha_{12}}(x-p_k) \,{\rm sign}(x-p_j)e^{-\eps^{1+\alpha/2}\sqrt{\mu}|x-p_j|}\, dx.
\end{align}
The Green's function is slowly varying and $\phi_0$ is well localized, moreover when $j=k$ the integral is zero due to even-odd parity.
We evaluate the integrals asymptotically, obtaining the expression
\beq
\dot{\vp} =\eps Q^{-1}\cA(\vp) \vq\,^\theta +O\!\left(\epsilon^{2-\alpha}, \eps\|W\|_X,\|W\|_X^2 \right),
\eeq
where $Q$ is the diagonal matrix of the amplitudes $\vq$ and the antisymmetric matrix $\cA(\vp)$ is defined component-wise as
\eqref{matrixA}. It is clear that the pulse evolution is at most $O(\eps)$, hence given  initial data of the form
\eqref{e:init-cond}, then the projected initial pulse condition $\vp_0$ is within $O(\eps^\alpha|\ln\eps|)$ of
$\vp$ and hence its distance to $\partial\cK$ is of the same order as  $d_0:=d(\vp,\partial \cK)$. Consequently the
time to arrive at $\partial\cK$ is $O(\eps^{-1}d_0)$. Moreover evolution of the form \eqref{Main-pulse} generically causes the
pulse spacing to increase in time. Indeed within the weak regime the interaction reduces to repulsive near-neighbor tail interaction,
and the semi-strong spectra is fixed, so that the time to exit the admissible configuration domain, $\cK$, is infinite. 
This completes the proof of Theorem\,\ref{thm:main}.

\section{Technical Estimates}
We prove several technical estimates used in Section 6.  The first involve the correction term $\Phi_c$ in the decomposition
\eqref{decompPhic}.

\begin{lemma}
For each admissible family $\cK$ of $N$-pulse configurations there exists $C>0$ such that 
\be
\label{vecPhi_1}\|\Phi_c\|_X+\|\nabla_\vp\, \Phi_c \|_X \leq C\eps^{1-\alpha/2},
\ee
for all $\vp \in \cK$,
\end{lemma}
\begin{proof}
From the estimates \eqref{adjointest}, the definition \eqref{pi-def} of $\pi_\vp$, and the bounds \eqref{3lemma2} we calculate
\[ \|\pi_\vp \cF(\Phi) \|_X \leq C \eps^2,\]
so that to the orders we are concerned with, $\tpi_\vp\cF(\Phi)=\cF(\Phi)$.  We take the $\|\cdot\|_X$ norm of \eqref{vecPhi1} and 
use the resolvent estimate \eqref{propres1} at $\lambda=0$, for
 which $\Re (k_0)=|k_0|=\eps^{1+\alpha/2}$,  to obtain
\be
\|{\Phi}_c\|_X  \leq C\eps^{-\alpha/2}\left(\eps \|\cF_1(\Phi)\|_{L^1} +\|\cF_2(\Phi)\|_{L_{1,\vp}^1}\right).
\ee
The bounds \eqref{3lemma2} on the residual yield the $\Phi_c$ estimate of \eqref{vecPhi_1}.
Taking $\nabla_\vp$ of \eqref{vecPhi1}, and observing that $L_{\vp_0}$ and $\tpi_{\vp_0}$ are independent of $\vp$, yields
\be
\nabla_\vp\,\Phi_c=-\tL_{\vp_0}^{-1} \tpi_{\vp_0}\nabla_{\vp}\, \cF(\Phi).
\ee
The estimate \eqref{propres1}, applied at $\lambda=0$, yields the bound
\beq
\|\nabla_\vp\,\Phi_c\|_X \leq C \left(\eps^{1-\alpha/2}\|\nabla_\vp \cF_1(\Phi)\|_{L^1} + \eps^{-\alpha/2}\|\nabla_\vp\cF_2(\Phi)\|_{L_{1,\vp}^1}\right),
\eeq
which, when combined with \eqref{3lemma2} yields the $\nabla_\vp$ estimate of \eqref{vecPhi_1}.\hfill
\end{proof}

We break $\Delta L$  into secular,  $\Delta_s L:=  L_\vp-L_{\vp_0},$ and
reductive, $\Delta_r L:= L_{\vp_0} - \tL_{\vp_0},$  which satisfy the following bounds.
\begin{lemma}\label{L:DeltaL}
For each admissible family, $\cK$,  of $N$-pulse configurations there exists $C>0$ such that for all $V=(V_1,V_2)^T\in X_\vp$
\begin{align}
\label{DeltaLsW}\eps \|[\Delta_s L\, V]_1\|_{L^1} + \|[\Delta_s L\, V]_2\|_{L_{1,\vp}^1} \leq& C T_\vp(t)\|V\|_X, \\
\label{DeltaLrW}\eps^2\|[\Delta_r L \, V]_1\|_{L^1_{1,\vp}}+\|[\Delta_r L\, V]_2\|_{L_{1,\vp}^1} \leq& C\eps \|V\|_X,\\
\label{DeltaLrmass} \left|\otimes \vec{\chi} \cdot [\Delta_r L\, V]_1\right| =& 0, \\
\label{nonlinestimate} \eps^{-1}\|\cNon_1(V)\|_{L^1}+\|\cNon_2(V)\|_{L_{1,\vp}^1} \leq& C \|V\|_X^2.
\end{align}
\end{lemma}
\begin{proof}
We first examine the secular operator $\Delta_s L$,
\beq
 \Delta L_s =\begin{pmatrix} \eps^{-1}\alpha_{11}\left(\Phi_{\vec{p},1}^{\alpha_{11}-1}\Phi_{\vec{p},2}^{\alpha_{12}}- \Phi_{\vec{p}_0,1}^{\alpha_{11}-1}\Phi_{\vec{p}_0,2}^{\alpha_{12}}\right) &
 \eps^{-1} \alpha_{12}\left(\Phi_{\vec{p},1}^{\alpha_{11}}\Phi_{\vec{p},2}^{\alpha_{12}-1}-\Phi_{\vec{p}_0,1}^{\alpha_{11}}\Phi_{\vec{p}_0,2}^{\alpha_{12}-1}\right) \\
 \alpha_{21}\left(\Phi_{\vec{p},1}^{\alpha_{21}-1}\Phi_{\vec{p},2}^{\alpha_{22}}-\Phi_{\vec{p}_0,1}^{\alpha_{21}-1}\Phi_{\vec{p}_0,2}^{\alpha_{22}}\right) & \alpha_{22}\left(\Phi_{\vec{p},1}^{\alpha_{21}}\Phi_{\vec{p},2}^{\alpha_{22}-1}-\Phi_{\vec{p}_0,1}^{\alpha_{21}}\Phi_{\vec{p}_0,2}^{\alpha_{22}-1}\right)
 \end{pmatrix}.
\eeq
The $\Phi_2$ terms  are smooth and decay exponentially away from each pulse position, in particular there exists a $C>0$ such that
\be
\|\Phi_{\vec{p},2}-\Phi_{\vec{p}_0,2}\|_{H^1}+\|\Phi_{\vec{p},2}-\Phi_{\vec{p}_0,2}\|_{L^1} \leq C|\vp - \vp _0|\leq C T_\vp(t),
\ee
where  $T_\vp$ is defined in \eqref{T2}. From  \eqref{dxPhi1Linfty} we have the estimates $\|\partial_x \Phi_1\|\leq C\eps$  while
for any fixed $\beta>0$, $\|\Phi_2^\beta\Phi_1\|_{L^\infty}\leq C$. Combining these estimates yields 
\beq
\|[\Delta_sL\, V]_1\|_{L^1} \leq \frac{C}{\eps}\left(\|\Phi_{\vec{p},2}- \Phi_{\vec{p}_0,2}\|_{L^1}\|V_1\|_{L^{\infty}} +
 \|\Phi_{\vec{p},2}- \Phi_{\vec{p}_0,2}\|_{L^2}\|V_2\|_{L^2}\right) \leq \frac{C}{\eps} T_\vp(t) \|V\|_X. 
\eeq
The estimate on $\|[\Delta_sL\, V]_2\|_{L_{1,\vp}^1}$ is similar and \eqref{DeltaLsW} follows.  For the reductive operators, $\Delta_rL$,  
the difference $L_{\vp_0}-\tL_{\vp_0}$ is large but has small mass in each $\vchi$ window, which permits the application of \eqref{propres3}.  
The reductive operator takes the form
\beq
\Delta_r L = \begin{pmatrix}  \epsilon^{-1}\left( \alpha_{11}\Phi_{\vec{p}_0,1}^{\alpha_{11}-1}\Phi_{\vec{p}_0,2}^{\alpha_{12}}-J_{11}(\vec{p}_0)\right) &
 \epsilon^{-1}\left(\alpha_{12}\Phi_{\vec{p}_0,1}^{\alpha_{11}}\Phi_{\vec{p}_0,2}^{\alpha_{12}-1}-J_{12}(\vec{p}_0)\right)\\
 \left(\alpha_{21}\Phi_{\vec{p}_0,1}^{\alpha_{21}-1}\Phi_{\vec{p}_0,2}^{\alpha_{22}}-J_{21}(\vec{p}_0)\right) &
  \alpha_{22}\Phi_{\vec{p}_0,1}^{\alpha_{21}}\Phi_{\vec{p}_0,2}^{\alpha_{22}-1}-\alpha_{22}\sum\limits_{j=1}^N  \phi_0^{\alpha_{22}-1}(x-p_{j,0})
  \end{pmatrix},
\eeq
where $J_{11}$ and $J_{12}$, defined in \eqref{J11}-\eqref{J12}, are chosen so that
\[\otimes \vchi\cdot [\Delta_r L\, V]_1=0, \]
for all $V\in W^{1,1}_\xi(\bR),$ and thus \eqref{DeltaLrmass} is satisfied.
The weighted-norm bound $\|[\Delta_r L\,V\|_{L^1_{1,\vp}}\leq C\eps^{-1}\|V\|_X$ appearing as the first term of \eqref{DeltaLrW}
follows from typical H\"older estimates. The second bound of  \eqref{DeltaLrW} follows from the form, \eqref{J21-def}, of $J_{21}$ and the
last estimate appearing in \eqref{dxPhi1Linfty} applied to $(\Phi_1-q_k)\phi_k^\beta$ for $\beta$ taking the values $\alpha_{22}-1$ and $\alpha_{22}$.

For the nonlinearity, we consider the case $\alpha_{12}=\alpha_{22}=2$; the other cases are similar. The leading order terms in the nonlinearity take the form,
\beq
\cNon(V)=\begin{pmatrix}  \eps^{-1}\Phi_1^{\alpha_{11}-1}\Phi_2 V_1V_2+\Phi_1^{\alpha_{11}}V_2^2 + \eps^{-1}\Phi_1^{\alpha_{11}-2}\Phi_2^2V_1^2 \\ 
 \Phi_1^{\alpha_{21}-1}\Phi_2 V_1V_2+\Phi_1^{\alpha_{21}}V_2^2 + \Phi_1^{\alpha_{21}-2}\Phi_2^2V_1^2
 \end{pmatrix},
 \eeq
where the derivatives in $V_1$ are smooth since $\Phi_1$ is uniformly bounded away from zero, where the powers $U^{\alpha_{j1}}$ in $\cF$ are non-smooth.
For terms in which the power of $\Phi_2$ is positive, we use the bound $\|\Phi_1^{\beta_1} \Phi_2^{\beta_2}\|_{L^\infty}\leq C$ 
which may be inferred from \eqref{dxPhi1Linfty}. For the terms without $\Phi_2$ we use \eqref{mega-bound}. Together this approach yields
\eqref{nonlinestimate}.
\hfill \end{proof}

Combining the estimates above with  Proposition\,\ref{P:semigroup} yields the following corollary.

\begin{corollary}
For each admissible class $\cK$ there exists $C>0$ such that for all $t>0$
\begin{align}
\label{semi11} \|S(t)\left(\tpi_{\vp_0}\Delta_sL\,V\right)\|_X \leq& CT_\vp(t) e^{-\epsilon^{\alpha}\nu t}\|V\|_X, \\
\label{semi12} \|S(t)\left(\tpi_{\vp_0}\Delta_r L \,V\right) \|_X \leq& C\eps e^{-\epsilon^{\alpha}\nu t}\|V\|_X, \\
\label{semi13} \|S(t)\left(\tilde{\pi}_{\vec{p}_0}\mathcal{N}(V)\right)\|_X \leq& Ce^{-\epsilon^{\alpha}\nu t}\|V\|_X^2.
\end{align}
\end{corollary}

\section{Discussion: Normal hyperbolicity and adiabatic stability}

The normal hyperbolicity condition $\alpha\in[0,\frac12)$ presented in Theorem\,\ref{thm:main}, can be motivated by the following simple argument. 
For a fixed $N$ pulse configuration $\vp_0=\vp(t_0)$ with associated linearization $L_{\vp_0}$, we obtain decay estimates
for the semi-group generated by $L_{\vp_0}$, see  Proposition\,\ref{P:semigroup},  which guarantee exponential  decay over time-scales  
$t-t_0\sim \eps^{-\alpha}$. However, over this time period we find that the pulse positions $\vp=\vp(t)$ experience a drift on the order of 
$|\vp(t_0+ t)-\vp_0|\sim \eps  t$ and the linearization about the evolving pulse configuration is 
time-dependent, $L:=L_{\vp(t)}$. For non-self adjoint operators, time-dependence in the linear operator can act as a source of forcing that 
destabilizes the underlying equilibrium, even if for each fixed $t$ the linearization has exponentially stable semi-groups. In general, the problem 
of characterizing the semi-group produced by a time-dependent linearization is nontrivial, \cites{Kato-70, Kato-73}. However, if the time-dependence of 
the linear operators is sufficiently slow compared to the exponential decay rate of the semi-group associated to each fixed-time operator, one would 
generically imagine that the time-dependent operator inherits the exponential decay of its frozen-time constituents.  

A characterization of normal hyperbolicity for a flow in a neighborhood of
 a manifold should specify how slow is ``slow-enough.'' We introduce a characterization via the ``secularity'' in the linearization, 
 $\Delta L :=L_{\vp(t)}-L_{\vp_0}$. Introducing the spectral projection, $\pi_{\vp_0}$, onto the tangent plane of the manifold of semi-strong 
 $N$-pulses at pulse configuration $\vp_0$,  the forcing induced by the secular term is characterized by the
  quantity $\|(I-\pi_{\vp_0})\Delta L\|$ where $\|\cdot\|$ denotes an appropriate operator norm. Assuming the generic scalings 
\[ \|(I-\pi_{\vp_0})\Delta L\| \sim \|\Delta L\|\sim |\vp(t_0+t)-\vp_0| \sim \eps t, \]
then over the time-scale, $t-t_0\sim \eps^{-\alpha},$ required to obtain exponential decay, the contribution of the secularity scales like
\[  \int _0^{\eps^{-\alpha}} \|(I-\pi_{\vp_0})\Delta L(s)\| e^{-\nu \eps^\alpha(\eps^{-\alpha}-s)}\, ds \sim \eps^{1-2\alpha} \int _0^1 s e^{-\nu(1-s)}\, ds=O(\eps^{1-2\alpha}).\]
This suggests the normal-hyperbolicity constraint $\alpha<\frac12$. In the study here-in, this constraint is manifest in equation \eqref{T1-est2}, 
which requires  $(\Delta t)^2\eps\ll 1$, where $\Delta t\sim \eps^{-\alpha}$ is the time required to obtain $O(1)$ decay in the semi-group.
 This discussion motivates the normal hyperbolicity conjecture: ``The square of the linear decay time multiplied by the pulse velocity 
 must be sufficiently smaller than one.'' For the system \eqref{UV}, this is the genesis of the normal hyperbolicity  condition, $\alpha<\frac12$. 
   
   What structure remains for decay rates slower than the normal hyperbolicity threshold?  Our intuition is that the RG iteration, which effectively produces linear semi-group estimates for the time-dependent linearization, $L_\vp(t)$, will fail, and the separation of the spectral spaces into dynamic (the $N$ translational eigenvalues)  and slaved (the rest of the spectrum) modes breaks down.  It may be possible to recover the required linear estimates if the nose of the essential spectrum, which extends to $-\eps^\alpha\mu$, is incorporated into the dynamic elements of the decomposition. In particular, resonance poles, the zeros of $\det(I+\cN_\lambda)$ for $\lambda$ in the essential spectrum (see Proposition 4.1), must be accounted for. It is unclear if this can be achieved in a finite-dimensional setting, or if the underlying dynamics would be truly infinite dimensional. The situation is evocative of the spectral gaps required to produce inertial manifolds for  dissipative PDEs, see \cite{FST-88}. It is precisely the spectral gap which is closing here.

\section*{Acknowledgments}
 K.P. acknowledges support from the National Science Foundation under grants  DMS 0708804 and DMS 1125231.

\end{document}